\documentclass[sn-mathphys,Numbered]{sn-jnl}% Math and Physical Sciences Reference Style
\usepackage{graphicx}%
\usepackage{multirow}%
\usepackage{amsmath,amssymb,amsfonts}%
\usepackage{amsthm}%
\usepackage{mathrsfs}%
\usepackage[title]{appendix}%
\usepackage{xcolor}%
\usepackage{textcomp}%
\usepackage{manyfoot}%
\usepackage{booktabs}%
\usepackage{algorithm}%
\usepackage{algorithmicx}%
\usepackage{algpseudocode}%
\usepackage{listings}%
\theoremstyle{thmstyleone}%
\newtheorem{theorem}{Theorem}%  meant for continuous numbers
\newtheorem{proposition}[theorem]{Proposition}% 

\theoremstyle{thmstyletwo}%
\newtheorem{example}{Example}%
\newtheorem{remark}{Remark}%

\theoremstyle{thmstylethree}%

\begin{document}

\title[Topology and regularity for generalized ultradistribution algebras]{Topology and regularity for generalized ultradistribution algebras}

%%=============================================================%%
%% Prefix	-> \pfx{Dr}
%% GivenName	-> \fnm{Joergen W.}
%% Particle	-> \spfx{van der} -> surname prefix
%% FamilyName	-> \sur{Ploeg}
%% Suffix	-> \sfx{IV}
%% NatureName	-> \tanm{Poet Laureate} -> Title after name
%% Degrees	-> \dgr{MSc, PhD}
%% \author*[1,2]{\pfx{Dr} \fnm{Joergen W.} \spfx{van der} \sur{Ploeg} \sfx{IV} \tanm{Poet Laureate} 
%%                 \dgr{MSc, PhD}}\email{iauthor@gmail.com}
%%=============================================================%%

\author[1]{\fnm{Stevan} \sur{Pilipovi\'{c}}}\email{pilipovic@dmi.uns.ac.rs}
\equalcont{These authors contributed equally to this work.}

\author[3]{\fnm{Dragana} \sur{Risteski}}\email{dragana.v.risteski@gmail.com}
\equalcont{These authors contributed equally to this work.}

\author[2]{\fnm{Dimitris} \sur{Scarpal\'ezos}}\email{dim.scarpa@gmail.com}
\equalcont{These authors contributed equally to this work.}

\author*[1]{\fnm{Milica} \sur{\v Zigi\'c}}\email{milica.zigic@dmi.uns.ac.rs}

\affil*[1]{\orgdiv{Faculty of Sciences}, \orgname{University of Novi Sad}, \orgaddress{\street{Trg D. Obradovica 4}, \city{21000 Novi Sad}, %\postcode{21000}, %\state{State}, 
\country{Serbia}}}

\affil[2]{\orgdiv{Centre de Math\'ematiques de Jussieu}, \orgname{Universit\'e Paris 7 Denis Diderot}, \orgaddress{\street{Case Postale 7012, 2, place Jussieu}, \city{F-75251 Paris Cedex 05}, 
%\postcode{75251}, %\state{State}, 
\country{France}}}

\affil[3]{%\orgdiv{Department}, \orgname{Organization}, 
\orgaddress{%\street{Street}, 
\city{21000 Novi Sad}, %\postcode{610101}, \state{State}, 
\country{Serbia}}}

%%==================================%%
%% sample for unstructured abstract %%
%%==================================%%

\abstract{Compiling essential results for non-quasianalytic ultradistribution spaces and Colombeau versions of generalized ultradistribution algebras, we analyze  strong $B$- and strong $R$-association of a generalized ultradistribution $[(f_\varepsilon)]$. The strong association of $[(f_\varepsilon)]$ to a Komatsu-type ultradistribution $T$, with additional assumption on regularity of $[(f_\varepsilon)]$  of Beurling, respectively, Roumieu type,  implies that $T$ is an ultradifferentiable function of Beurling, Roumieu type, respectively.  We demonstrate that, under suitable conditions on regularity, a weakly negligible net $(g_\varepsilon)_{\varepsilon\in(0,1)}$ (meaning that the net of complex numbers $(\int g_\varepsilon\phi dx)_{\varepsilon\in(0,1)}$ is Beurling, respectively, Roumieu negligible for every ultradifferentiable function $\phi$ in the corresponding test space), is a negligible net in the sense of generalized ultradistributions. Furthermore, we prove that a translation invariant generalized ultradistribution $g$ is equal to a generalized constant in both types of generalized ultradistribution algebras. }



\maketitle
 
%%%%%%%%%
\section{Introduction and basic notions}
%%%%%%%%%
 
Spaces of ultradistributions of non-quasianalytic type, embedded into corresponding algebras of generalized ultradistributions  $\mathcal{G}^{(M_p)} (\Omega)$ and $\mathcal{G}^{\{M_p\}} (\Omega),$  are introduced in \cite{D.V.1, ps, Del-Has-Pil-Val, DHPV}  in a manner that respects  products of ultradifferentiable functions and the action of ultradifferential operators. Furthermore, it was demonstrated in \cite{D.V.1}, \cite{ps} and \cite{Del-Has-Pil-Val} that a generalized ultradistribution  is regular if it is embedded into  a specific algebra of regular elements.  The term "algebra of regular elements" implies that this algebra corresponds to an appropriate algebra of smooth functions.

In this paper, inspired by the methodology employed in establishing Hausdorff topologies in Colombeau-type algebras \cite{Del-Sca-3,Sca1,Sca3} (also referenced in \cite{ober 001}),  we introduce topological sharp structures in the algebras $\mathcal{G}^{(M_p)} (\Omega)$ and $\mathcal{G}^{\{M_p\}} (\Omega),$ as well as in their vector-valued versions. 

In Colombeau theory (refer to \cite{col1, gkos, her-kunz, Sca1, Sca2, Sca3, Sca4, psv.1}), it has been established that a property known as strong association between a regular generalized function $[(f_\varepsilon)]$ (where $[(f_\varepsilon)] \in \mathcal{G}^{\infty}(\Omega)$) and a distribution $F \in \mathcal{D}'(\Omega)$ implies that $F$ is a smooth function.
%Later analogous theorems have been proved for other kinds of regularity under the hypothesis of "strong %associations" (\cite{PSV1}\cite{PS2} \cite{PSV2})and even in a first paper on ultradistributions \cite{PS3}
We  prove these regularity results within the framework of generalized ultradistributions.  Actually, distinctions arise between the Beurling and Roumieu cases, prompting the need for an additional assumption in $R$-strong association to establish the corresponding regularity result. Example \ref{e1} and Proposition \ref{RUM} in Section \ref{s3} are related to the $R$-strong association. 

In the final section, we analyse Beurling and Roumieu weakly negligible nets, introduced and explained in Section \ref{sec 1.4}. Generalized negligible nets of complex numbers $(c_\varepsilon)_{\varepsilon\in(0,1)}$ in the sense of ultradistributions are defined through the estimate $|c_\varepsilon|=O({e^{-M(k/\varepsilon)}})$ for every $k\in\mathbb R$ (Beurling case) and $|c_\varepsilon|=O({e^{-N_{(k_i)}(1/\varepsilon)}})$ for every $(k_i)_{i\in \mathbb{Z}+}\in\mathfrak R$ (Roumieu case). Here, $M$ and $N_{(k_i)}$ are associated functions for the sequence $(M_p)_{p\in\mathbb{N}}$ determining both types of ultradistribution spaces, and $(k_i)_{i\in\mathbb{Z}_+}$ is a sequence of positive numbers increasing to infinity.  Weakly negligible nets $(g_\varepsilon)_{\varepsilon\in (0,1)}$ are those for which the nets of complex numbers ($\langle g_\varepsilon,\phi\rangle)_{\varepsilon\in(0,1)}$ are Beurling or Roumieu negligible, for  every $\phi$ in the corresponding test space. The property that a net $(g_\varepsilon)_{\varepsilon\in(0,1)}$ is weakly negligible, with additional assumptions, implies that this net is a  negligible net of generalized ultradistribution, that is, $[(g_\varepsilon)]=0$. This is the focus of Theorems \ref{thm 4.1} and \ref{4-2}.  In the last theorem, we establish that the translation invariant generalized ultradistribution $g$ is equal to a generalized constant in both types of generalized ultradistribution algebras.  While our proofs draw inspiration from our previous paper \cite{PSV}, they are adapted to the complexities of the current setting.

In terms of notation, the big "$O$" symbol adheres to its standard meaning: we express  $a_n=O(n^s)$   if $|a_n/n^s|\leq C, $ $n\geq n_0,$  for some $C>0$  and $n_0\in\mathbb N$; similarly,  $a_\varepsilon=O(\varepsilon^s)$ if  $|a_\varepsilon/\varepsilon^s|\leq C, $ $\varepsilon\in (0,\varepsilon_0)$, for some $C>0$ and $\varepsilon_0\in(0,1).$ For brevity, sequences such as $(M_p)_{p\in\mathbb{N}}, $ $(r_i)_{i\in\mathbb{Z}_+},$ or nets $(f_\varepsilon)_{\varepsilon\in (0,1)}$ will be abbreviated as $(M_p), $ $(r_i),$ and $(f_\varepsilon),$ respectively.

In the paper, we employ the following version of the Baire theorem: Suppose $S$ is either a complete metric space or a locally compact $T_2$-space.  If $E_i,$ $i\in \mathbb{N}$ constitutes a countable collection of nowhere dense subsets of $S$ (i.e., $ \overline{E_i}^\circ=\emptyset$), then $S$ cannot be expressed as the union of the sets $E_i,$ $i\in \mathbb{N}.$

\subsection{Notation}\label{bas-def}

We employ the standard notations: $\mathbb N=\mathbb Z_+\cup\{0\}$ and $D^{\alpha_j}_{x_j}=i^{-\alpha_j}\partial^{\alpha_j}/\partial x_j^{\alpha_j}, $ $j=1,\ldots,d, $ $\alpha=(\alpha_1,\dots,\alpha_d) \in\mathbb N^d$.   We assume that $M_p, $ $p\in\mathbb N,$ is a sequence of positive numbers, $M_0=1$, %$M_p^{1/p}\rightarrow\infty$, as $p\rightarrow \infty$, 
so that it satisfies conditions (cf. \cite{Komatsu1}):
\begin{itemize}
\item[(M.1)] $M_{p}^{2} \leq M_{p-1} M_{p+1},$ $p \in\mathbb Z_+;$
\item[(M.2)] $M_{p} \leq c_0H^{p} \min_{0\leq q\leq p} \{M_{p-q} M_{q}\},$  $p,q\in \mathbb N,$ for some $c_0,H\geq 1;$
\item[(M.3)] $\sum_{j=p+1}^{\infty}M_{j-1}/M_j\leq c_0pM_p/M_{p+1},$ $p\in\mathbb N,$ for some $c_0\geq 1.$
\end{itemize}
%\noindent Sometimes we will impose the following additional condition on a weight %sequence $(M_p)_{p\in\mathbb N}$\\
%\indent $(M.3)'$ $\sum_{p=1}^{\infty}M_{p-1}/M_p<\infty$,\\
%\noindent or the stronger condition\\
%we will always emphasise when we impose either one of these conditions. 
The Gevrey sequence $p!^{\sigma},$ $p\in\mathbb N,$ satisfies (M.1) and (M.2) when $\sigma>0;$ if $\sigma>1,$ it also satisfies (M.3).  We  use the notation $M_{\alpha}$ for $M_{|\alpha|},$ $|\alpha|=\alpha_1+\dots+\alpha_d.$
% If $(A_p)_{p\in\mathbb N}$ is another weight sequences, the notation $A_p\subset M_p$ means %$A_p\leq CL^pM_p$, $p\in\mathbb N$, for some $C,L>0$. 

%%
\subsection{Associate and subordinate functions}

As in \cite{Komatsu1},  the associated function for the weight sequence $(M_p)$ is defined by
\begin{equation}\label{assoc-fun-sec}
M(\rho):=\sup_{p\in\mathbb N}\ln_+  (\rho^{p}/M_{p}),\quad \rho> 0.
\end{equation}
It is a non-negative, continuous, monotonically non-decreasing function, which vanishes for sufficiently small $\rho>0$ and increases more rapidly than $(\ln \rho)^s$, when $\rho$ tends to infinity, for any $s\in\mathbb N$. %When $M_p=p!^{\sigma}$, $\sigma>0$, we have $M(\rho)\asymp \rho^{1/\sigma}$.% Employing the inequality $(\lambda+\rho)^p\leq 2^p\lambda^p+2^p\rho^p$, $\lambda,\rho\geq 0$, $p\in\mathbb Z_+$, one easily verifies the bound
By \cite{Komatsu1}, one can deduce that condition (M.2) implies
\begin{equation*}\label{B-c}
\forall h>0, \;\exists k_1>0, \;\exists k_2>0, \quad  k_1M(\rho)\leq M(h \rho)\leq k_2 M(\rho), \;\;\rho>0.
\end{equation*}

By $\mathfrak{R}$ is denoted the set of all positive monotonically increasing sequences $r_i,$ $i\in\mathbb Z_+$, in the sequel denoted as $(r_i),$ such that $r_i\rightarrow \infty$, as $i\rightarrow \infty$. With the partial order relation $(r_i)\leq (k_i)$ if $r_i\leq k_i$,  $\forall i\geq i_0$ for some $i_0\in\mathbb Z_+$,  $(\mathfrak{R},\leq)$ becomes a directed set (both upwards and downwards directed). 
%For $(r_i)\in\mathfrak{R}$, we denote by $N_{(r_i)}(\cdot)$ the associated function to the sequence $N_p=M_p\prod_{i=1}^pr_i$, $p\in\mathbb N$ 
Although, for given $(r_i)\in\mathfrak{R},$ a sequence $N_p=M_p\prod_{i=1}^pr_i$,  $p\in\mathbb N$, may fail to satisfy (M.2), we can still define its associated function as in \eqref{assoc-fun-sec}.
%Let  $(r_i)\in\mathfrak{R}$. Then the  sequence $(N_p)$ is defined by $N_0=1$, $N_p=M_p\prod_{i=1}^p r_i$, $p\in\mathbb N$ 
The associated function corresponding to the sequence $(N_p)$ is denoted by $ N_{(r_i)}(\rho)=\sup_{p\in\mathbb N}\ln_{+}({\rho^p}/{N_p}),$ $\rho>0.$ Note,%  that for given $(r_i)$ and every $k > 0$ there is $\rho_0 > 0$ such that 
\begin{align*}
\forall (r_i)\in \mathfrak{R},\,\forall k>0, \,\exists \rho_0>0,\;\; N_{(r_i)}(\rho)\leq kM(\rho), \;\rho > \rho_0.
\end{align*}
Moreover,
% for every $(r_i)\in\mathfrak{R}, c_1>0, c_2>0$ there exist $\tilde{r}_i$ and $\rho_0>0$ such that 
\begin{equation*}\label{Ra}
\forall (r_i)\in\mathfrak{R}, \,\forall c_1,c_2>0, \,\exists (\tilde{r}_i)\in\mathfrak R,\;\; N_{(r_i)}(c_1\rho)\leq c_2N_{(\tilde r_i)}(\rho),\; \rho>0.
\end{equation*}
 
We denote by  $c$  a subordinate function  corresponding to $(r_i)\in\mathfrak R$, i.e. to $N_p=M_p\prod_{i=1}^pr_i,$ $p\in\mathbb{N},$ defined by $M(c(\rho))=N_{(r_i)}(\rho), $ $\rho>0.$ We will also use notation $c_{(r_i)}(\rho),$ $\rho>0$ to  distinguish subordinate functions which correspond to different $(r_i)\in \mathfrak{R}.$ Similarly,  we use notation $\tilde c(\rho)$ to denote a subordinate function for $(\tilde r_i)\in \mathfrak{R}$ (so  $c_{(\tilde r_i)}(\rho)=\tilde c(\rho),$ $\rho>0,$ $M(\tilde c(\rho))=N_{(\tilde r_i)}(\rho)$).

Important  properties of  subordinate functions are given in \cite{Komatsu1},  see also  \cite{KMMP}.
We note,
\begin{gather*}
c_{(k_i)}(\rho/k)=c_{(k_i/k)}(\rho),\;\rho>0,\;k>0,\\
(b_i)\leq (k_i)\quad \Rightarrow\quad N_{(k_i)}(\rho)\leq N_{(b_i)}(\rho)\quad\Rightarrow\quad c_{(k_i)}(\rho)\leq c_{(b_i)}(\rho),\quad \rho>0.
\end{gather*}
%\begin{equation*}
%\forall (r_i)\in\mathfrak{R}, \,\exists k_1>0, \,\exists k_2>0, \,\exists (\tilde r_i)\in\mathfrak{R}, \; \;k_1\tilde c(\rho) \leq c(\rho) \leq k_2 \tilde c(\rho),\; \rho>0.
%\end{equation*}
%($c(\rho)=c_{r_i}(\rho), \tilde c(\rho)=c_{\tilde r_i}(\rho)$).  
%Thus,
%\begin{equation*}\label{Res}
%\forall (r_i)\in\mathfrak{R},  \,\exists s_1>0, \exists s_2>0, \,\exists (\tilde r_i)\in\mathfrak{R},\;\;
%s_1 N_{(\tilde r_i)}(\rho)\leq N_{(r_i)}(\rho)\leq s_2 N_{(\tilde r_i)}(\rho), \;\rho>0. 
%\end{equation*}
%\textbf{NB.  za ovo u zagradi $c(\rho)=c_{r_i}(\rho), \tilde c(\rho)=c_{\tilde r_i}(\rho)$ nisam sigurna \v sta je, jer nismo imali oznaku $c_{r_i}$.}

%%
\subsection{Test spaces and spaces of ultradistributions}
\label{sec1.3}

Let $\Omega\subset \mathbb{R}^d$ be an open set.  We consider in the whole paper compact subsets of $\Omega,$ $K\subset\subset \Omega,$ so that they are regular compact sets (cf. \cite{Komatsu1}).  The space $\mathcal{E}^{M_p,h}(K), $ $h>0$, respectively,  $\mathcal{E}^{M_p,(r_i)}(K), $ $(r_i)\in\mathfrak R,$ consists of all $\phi\in C^\infty(\Omega)$ such that
\begin{equation*}
\|\phi\|_{\mathcal{E}^{M_p,h}(K)}=\sup_{x\in K,\alpha\in \mathbb{N}^d}\frac{|\phi^{(\alpha)}(x)|}{h^{|\alpha|} M_{\alpha}}<\infty,
\end{equation*}
respectively,
\begin{equation*}
\|\phi\|_{\mathcal{E}^{M_p,(r_i)}(K)}=\sup_{x\in K,\alpha\in \mathbb{N}^d}\frac{|\phi^{(\alpha)}(x)|}{M_{\alpha}\prod_{i\leq|\alpha|}r_i}<\infty.
\end{equation*}

Since we deal always with the sequence $(M_p)$, we will use below notation $\|\phi\|_{K,h}$, respectively $\|\phi\|_{K,(r_i)}$.  If it is crucial to emphasize the use of the sequence $(M_p),$ we will denote it as $\|\phi\|_{K,M_p,h}$ and  $\|\phi\|_{K,M_p,(r_i)}$ for the respective cases. The space of $(M_p)$-, respectively,  $\{M_p\}$-ultradifferentiable functions is defined as
\begin{equation*}
\mathcal{E}^{(M_p)}(\Omega)=\varprojlim_{K\subset\subset\Omega}\varprojlim_{h\to 0}\mathcal{E}^{M_p,h}(K),\; \text{respectively, }\;
\mathcal{E}^{\{M_p\}}(\Omega)=\varprojlim_{K\subset\subset\Omega}\varprojlim_{(r_i)\in\mathfrak R}\mathcal{E}^{M_p,(r_i)}(K).
\end{equation*}
Note that for the definition of the spaces given above we need only conditions (M.1) and (M.2), while for the next ones, with compactly supported test functions, we use (M.3).  Actually a weaker condition (M.3)' is sufficient but since we will use the so-called second structural theorem \cite[p.~91]{Komatsu1}, the condition (M.3) is assumed.

The subspace of $\mathcal{E}^{M_p,h}(K)$, respectively, $\mathcal{E}^{M_p,(r_i)}(K)$, consisting of elements with supports in $K\subset\subset\Omega$ is denoted as $\mathcal{D}^{M_p,h}(K)$, respectively, $\mathcal{D}^{M_p,(r_i)}(K)$ and we set
\begin{align*}
\mathcal{D}^{(M_p)}(K)=\varprojlim_{h\to 0}\mathcal{D}^{M_p,h}(K),&\quad 
\mathcal{D}^{\{M_p\}}(K)=\varprojlim_{(r_i)\in\mathfrak{R}}\mathcal{D}^{M_p,(r_i)}(K),
\\
\mathcal{D}^{(M_p)}(\Omega)=\varinjlim_{K\subset\subset\Omega}\mathcal{D}^{(M_p)}(K),&\quad 
\mathcal{D}^{\{M_p\}}(\Omega)=\varinjlim_{K\subset\subset\Omega}\mathcal{D}^{\{M_p\}}(K).
\end{align*}
%Naturally, the non-triviality of these spaces of compactly supported functions is equivalent %to $(M.3)'.$ 

Their duals $\mathcal{D}'^{(M_p)}(\Omega)$ and $\mathcal{D}'^{\{M_p\}}(\Omega)$ are the ultradistribution spaces of class $(M_p)$ (Beurling type) and class $\{M_p\}$ (Roumieu type), respectively. We refer to \cite{Komatsu1,Komatsu2,KMMP} for the properties of these spaces. 

We denote by $\mathcal{S}^{M_p,h},$ $h>0,$ the Banach space of all $\varphi\in C^\infty(\mathbb{R}^d)$ for which the norm
\begin{align*}
\sup_{\alpha\in \mathbb{N}^d}\frac{h^{|\alpha|}\|e^{M(h|\cdot|)}D^\alpha\varphi\|_{L^\infty(\mathbb{R}^d)}}{M_\alpha}<\infty,
\end{align*}
and define the space of tempered ultradifferentiable functions of class $(M_p),$ respectively, $\{M_p\},$ as  $\mathcal{S}^{(M_p)}(\mathbb{R}^d)=\varprojlim\limits_{h\to \infty}S^{M_p,h},$ respectively, $\mathcal{S}^{\{M_p\}}(\mathbb{R}^d)=\varinjlim\limits_{h\to 0}S^{M_p,h}.$ 
Actually, we will use the equivalent definition,
\begin{equation*}
    \mathcal{S}^{\{M_p\}}(\mathbb R^d)=\varprojlim_{(r_i)\in\mathfrak{R}}\mathcal{S}^{M_p,(r_i)}.
\end{equation*}
As a custom, one employs the notation $*$ jointly for $(M_p)$ and $\{M_p\}$ to treat both cases simultaneously. 

Corresponding strong dual $\mathcal{S}'^*(\mathbb{R}^d)$ is the spaces of tempered ultradistributions. The Fourier transform is a topological isomorphism on $\mathcal{S}^*(\mathbb{R}^d)$ and on $\mathcal{S}'^*(\mathbb{R}^d)$.

We recall (see \cite{Komatsu1} and also \cite{KMMP}) that $P(D)$ is an $(M_p)$-, respectively, $\{M_p\}$-ultradifferential operator if it is of the form: $P(D)=\sum_{\alpha\in\mathbb{N}^d}a_\alpha D^\alpha,$ $a_\alpha\in\mathbb{C},$ where $|a_\alpha|\leq CL^{|\alpha|}/M_\alpha,$ $\alpha\in \mathbb{N}^d,$ for some $L>0$ and $C>0,$ respectively, for every $L>0$ there exists $C>0$.  Then the symbol $P(\xi)=\sum_{\alpha\in\mathbb{N}^d}a_\alpha \xi^\alpha,$ $\xi\in\mathbb{C}^d$ is the corresponding $(M_p)$-, respectively, $\{M_p\}$-ultrapolynomial.  

Following classes of Beurling, respectively, Roumieu type ultradifferential operators are important for Theorem \ref{4-2} of last section.
%\begin{equation}\label{11}
% P_r(D)=(1+D_1^2+\ldots+D_d^2)^l\prod_{p=1}^\infty\left(1+\frac{D^2_1+\ldots +D^2_d}{r^2p^{2t}}\right)=\sum_{p=0}^\infty a_p D^p, 
%\end{equation}
\begin{equation*}\label{z}
P_r(D)=(1+D_1^2+...+D_d^2)^l\prod_{i=1}^\infty\left(1+\frac{D^2_1+... +D^2_d}{r^2m_i^{2}}\right)=\sum_{\alpha\in\mathbb{N}^d}a_\alpha D^\alpha, \quad r>0, \; l\geq 0,
\end{equation*}
respectively,  of the form 
%\begin{equation}\label{22} 
%P_{r_p}(D)=(1+D_1^2+\ldots+D_d^2)^l\prod_{p=1}^\infty\left(1+\frac{D^2_1+\ldots +D^2_d}{r_p^2p^{2t}}\right)=\sum_{p=0}^\infty a_p D^p,
%\end{equation}
\begin{equation*}\label{zv}
P_{(r_i)}(D)=(1+D_1^2+...+D_d^2)^l\prod_{i=1}^\infty\left(1+\frac{D^2_1+... +D^2_d}{r_i^2m_i^2}\right)=\sum_{\alpha\in\mathbb{N}^d}a_\alpha D^\alpha, \;\;(r_i)\in\mathfrak R, \;l\geq 0,
\end{equation*}
where $m_i=M_i/M_{i-1},$ $i\in\mathbb{Z}_+.$ Then  the  corresponding ultrapolynomials (with $\xi_i$ instead of $D_i$) satisfy, in the Beurling case, 
\begin{align*} %\label{berling}
&\exists C, C_1, C_2>0, \;\exists h, h_1, h_2>0,\\
Ce^{hM(|\xi|)}\leq |P_r(\xi)|&\leq C_1 e^{ h_1M(|\xi|)},  \;\; \xi\in\mathbb R^d 
\;\mbox{ and }\; |a_\alpha|\leq C_2 h_2^{|\alpha|}/M_\alpha, \;\;\alpha\in\mathbb N^d.
\end{align*}
%By \cite[Lemma 3.10]{Komatsu}, we know that $N_{(r_p)}(|\xi|)=M(c(|\xi|))$, where $c(|\xi|), \ |\xi|>0$ is a subordinate function which corresponds to $(r_p)_p\in\mathcal{R}$. 
We know that  for given $(r_i)\in\mathfrak{R}$  and its subordinate function  $c(|\xi|),$ $\xi\in \mathbb R^d,$ there exists $C_1>0$ such that 
\begin{equation*}\label{vazno6}
C_1e^{M(c(|\xi|))}\leq |P_{(r_i)}(\xi)|, \quad\xi\in\mathbb R^d.
\end{equation*}

\subsection{Spaces of generalized ultradistributions} 
\label{sec 1.4}

Algebras of generalized ultradistributions in the non-quasianalytic case were considered by several authors  \cite{D.V.1,Del-Has-Pil-Val,DHPV}. In order not to make confusion, in the sequel  we follow the definitions in \cite{D.V.1}.  This approach gives a clear embedding results for classical ultradistribution spaces.  Recall, in the Beurling case,  ultramoderate nets are elements of
\begin{align*}
\mathcal{E}^{(M_p)}_M(\Omega) = \Big\{(g_{\varepsilon}) \in (\mathcal{E}^{(M_p)}(\Omega))^{(0,1)}\,:&\, \forall K\subset \subset \Omega, \forall h>0,  \exists k\geq 0, \\
& \| g_{\varepsilon}\|_{K,h} =O(e^{{M}({k}/{\varepsilon})})\Big\},
\end{align*}  
while the negligible nets are elements of
\begin{align*}
\mathcal{N}^{(M_p)}(\Omega) = \Big\{(g_{\varepsilon}) \in  (\mathcal{E}^{(M_p)}(\Omega))^{(0,1)}\,: &\,\forall K\subset \subset \Omega, \forall h>0, \forall k \geq 0,\\
&\|g_{\varepsilon}\|_ {K,h} = O(e^{-{M}({k}/{\varepsilon})} )\Big\}.
\end{align*}  
In the Roumieu case, % the definitions are the following
\begin{align*}
\mathcal{E}^{\{M_p\}}_M(\Omega) =  \Big \{(g_{\varepsilon}) \in  (\mathcal{E}^{\{M_p\}}(\Omega))^{(0,1)}: & \; \forall K\subset \subset \Omega, \; \forall (h_i)\in\mathfrak R,   \; \exists (k_i)\in\mathfrak R, \\
 &\| g_{\varepsilon}\|_{K,(h_i)} =O(e^{N_{(k_i)}(1/{\varepsilon})})\Big\},
\end{align*}
\begin{align*}
\mathcal{N}^{\{M_p\}}(\Omega) =  \Big\{(g_{\varepsilon}) \in  (\mathcal{E}^{\{M_p\}}(\Omega))^{(0,1)}: & \; \forall K\subset \subset \Omega, \; \forall (h_i)\in\mathfrak R, \;  \forall (k_i)\in\mathfrak R,\\
&   \| g_{\varepsilon}\|_{K,(h_i)} =O(e^{-N_{(k_i)}(1/{\varepsilon})})\Big\}.
\end{align*} 
Algebras of generalized $(M_p)$-, respectively, $\{M_p\}$-ultrabistributions are defined by 
\begin{align*}
\mathcal{G}^{(M_p)} (\Omega) =   \mathcal{E}^{(M_p)}_M (\Omega) /\mathcal{N}^{(M_p)}(\Omega), \quad \text{respectively,} \quad \mathcal{G}^{\{M_p\}} (\Omega) =   \mathcal{E}^{\{M_p\}}_M (\Omega) /\mathcal{N}^{\{M_p\}}(\Omega).
\end{align*}
Let $(E,(\mu_p))$ be a Fr\'echet vector space,  topologized by a family of seminorms $\mu_p ,$ $p\in  \mathbb N.$ We additionally assume that it is an algebra; so, $E$ is a Fr\'echet algebra. 
% In the definitions which  are to follow, we consider smooth functions $g_\varepsilon,$ $%\varepsilon\in(0,1)$ with values in $E,$ denoted by $g_\varepsilon\in C^\infty(\Omega,E),%$ $\varepsilon\in (0,1).$

Beurling case:
\begin{align*}
\mathcal{E}^{(M_p)}_M[C^\infty(\Omega,E)] = 
\Bigg \{(g_\varepsilon) \in (C^\infty&(\Omega,E))^{(0,1)}\,: \, \forall K\subset \subset \Omega,  \,\forall p\in\mathbb N, \,\forall h>0,  \,\exists k>0,\\
& \|g_\varepsilon\|_{K,h,p}=\sup_{x\in K, \alpha\in \mathbb{N}^d}\frac{\mu_p(g^{(\alpha)}_{\varepsilon}(x))}{h^{|\alpha|} M_\alpha} =O(e^{M(k/\varepsilon)})\Bigg\};
\end{align*}  
while the negligible nets are elements of % defined by
\begin{align*}
\mathcal{N}^{(M_p)}[C^\infty(\Omega,E)] =  \Big\{(g_\varepsilon) \in (C^\infty& (\Omega,E))^{(0,1)}: \, \forall K\subset \subset \Omega, \,\forall p\in\mathbb N, \,\forall h>0, \, \forall k>0,\\
&\|g_\varepsilon\|_{K,h,p}=O(e^{-M(k/\varepsilon)})\Big\};
\end{align*}  
Roumieu case:
\begin{align*}
&\mathcal{E}^{\{M_p\}}_M[C^\infty(\Omega,E)] =\Bigg \{(g_{\varepsilon}) \in (C^\infty(\Omega,E))^{(0,1)}:\;\; \forall K\subset \subset \Omega, \,\forall p\in\mathbb N, \\
& \forall (h_i)\in\mathfrak R, \,\exists (k_i)\in\mathfrak R, \;\; \|g_\varepsilon\|_{K,(h_i),p}=\sup_{x\in K, \alpha\in \mathbb{N}^d}\frac{\mu_p(g^{(\alpha)}_{\varepsilon}(x))}{M_\alpha \prod_{i\leq |\alpha|}h_i} =O(e^{N_{(k_i)}(1/{\varepsilon})})\Bigg\};
\end{align*}  
while the negligible nets are elements of % defined by
\begin{align*}
\mathcal{N}^{\{M_p\}}& [C^\infty(\Omega,E)] =\Big \{(g_{\varepsilon}) \in (C^\infty(\Omega,E))^{(0,1)}: \\
& \forall K\subset \subset \Omega,  \forall p\in\mathbb N,  \forall (h_i)\in\mathfrak R, \forall (k_i)\in\mathfrak R, \;\;\|g_\varepsilon\|_{K,(h_i),p} =O(e^{-N_{(k_i)}(1/{\varepsilon})})\Big\}.
\end{align*}  
Generalized algebras are defined by 
\begin{align*}
\mathcal{G}^{(M_p)} [C^\infty(\Omega,E)] & =   \mathcal{E}^{(M_p)}_M [C^\infty(\Omega,E)] /\mathcal{N}^{(M_p)}[C^\infty(\Omega,E)],\\
\mathcal{G}^{\{M_p\}} [C^\infty(\Omega,E)] & =   \mathcal{E}^{\{M_p\}}_M [C^\infty(\Omega,E)] /\mathcal{N}^{\{M_p\}}[C^\infty(\Omega,E)].
\end{align*}
As in usual Colombeau algebra, the class of a net $(g_{\varepsilon})$ is denoted as $[(g_{\varepsilon})]$.
 
\begin{remark} \label{tc}
If $E=\mathbb C$, one obtains algebras   of Beurling  and Roumieu generalized ultradistributions.

Constant functions  $(\varepsilon,x)\mapsto a_\varepsilon\in \mathbb{C}, $ $\varepsilon\in(0,1),$  $x\in \Omega,$  are elements of $C^\infty(\Omega)$.  So, one obtains $ \mathbb{C}^{(M_p)}$, respectively, $ \mathbb{C}^{\{M_p\}}$,  Beurling, respectively, Roumieu type generalized complex numbers (see \cite{DHPV,D.V.1}), denoted as $\mathbb C^*=\mathcal E^*_M/\mathcal N^*_\mathbb{C},$ where $\mathcal E_M^{(M_p)}$, respectively, $\mathcal E_M^{\{M_p\}}$ consists of nets  $(c_\varepsilon)$ with the property $|c_\varepsilon|=O(e^{M(k/\varepsilon)}),$ respectively, $|c_\varepsilon|=O(e^{N_{(k_i)}(1/\varepsilon)}),$ for some $k\in \mathbb R$, respectively, some $(k_i)\in\mathfrak R;$  $\mathcal N^{(M_p)}_\mathbb{C}$, respectively, $\mathcal N^{\{M_p\}}_\mathbb{C}$ consists of nets  $(c_\varepsilon)$ with the property $|c_\varepsilon|=O(e^{-M(k/\varepsilon)}),$ respectively,  $|c_\varepsilon|=O(e^{-N_{(k_i)}(1/\varepsilon)}),$ for every $k\in\mathbb R$, respectively, every $(k_i)\in\mathfrak R.$  Real version is $ \mathbb{R}^{(M_p)}$, respectively, $ \mathbb{R}^{\{M_p\}}$.
\end{remark}

%%%%%%%%%%%% 
\section{Sharp topologies}
%%%%%%%%%%%

%%
\subsection{Beurling case} \label{BC}

In order to define the Hausdorff topology on $\mathcal{G}^{(M_p)} (\Omega),$ we first  note that
\begin{align*}
\bigcap_{K\subset\subset \Omega,k>0,h>0} \left\{(g_{\varepsilon})\,:\,\|g_{\varepsilon}\|_{K,h} = O(e^{-{M}({k}/{\varepsilon})})\right\}=\mathcal N^{(M_p)}(\Omega).
\end{align*}
Thus,  the basis of the zero in $\mathcal{G}^{(M_p)} (\Omega)$ consists of sets
\begin{equation*}
B_{K,h}^k = \left\{[(g_{\varepsilon})]: \|g_{\varepsilon}e^{{M}({k}/{\varepsilon})}\|_{K,h} =O(1)\right\}, \quad K\subset\subset \Omega , \;\;h>0 ,\;\;k>0. 
\end{equation*}
Choosing an increasing sequence of compact sets $K$ so that the union of their interiors equals $\Omega$ and  decreasing sequences of $h>0$ and $k>0,$ one obtains a countable basis of neighbourhoods of zero and the corresponding sharp topology. In particular, $\mathcal{G}^{(M_p)} (\Omega)$ is a topological algebra over the ring   $\mathbb{C}^{(M_p)}$. 

More generally, we define the basis of neighbourhoods in $\mathcal G^{(M_p)}[C^\infty(\Omega,E)]$,  where $E$ is a Fr\'echet algebra topologized by a family $\mu_p, $ $p\in \mathbb N,$  so that in the above definition $\|g_{\varepsilon}\|_{K,h},$ $\varepsilon\in (0,1)$ are replaced by $ \|g_\varepsilon\|_{K,h,p},$ $\varepsilon\in(0,1), $ $K\subset\subset \Omega ,$ $h>0, $ $ p\in \mathbb N.$ So we obtain the basis of neighbourhoods
\begin{align*}
B_{K,h}^{k,p} = \left\{[(g_{\varepsilon})]: \|g_{\varepsilon}e^{{M}({k}/{\varepsilon})}\|_{K,h,p} =O(1)\right\}, \quad K\subset\subset \Omega ,\;\; p\in \mathbb N, \;\;h>0 ,\;\;k>0. 
\end{align*}
    
In the same way one defines the sharp Hausdorff topology on the algebra of Beurling generalization functions with values in the Fr\'echet algebra $E$ by the use of representatives and the basis of neighbourhood for the representatives.
%$ \{[(g_{\varepsilon})]: \|g_{\varepsilon}e^{{M}({k}/{\varepsilon})}\|_{K,h,i}.$  
As previously, we have a countable basis of neighbourhoods of zero.  %We have the following proposition:
  
%\begin{proposition}\label{B-pro}
% $\mathcal{G}^{(M_p)}[C^\infty(\Omega,E)]$ is a topological module over $\mathcal{G}^{(M_p)}(\Omega)$.  Moreover,  $\mathcal{G}^{(M_p)}[C^\infty(\Omega,E)]$ is  a topological algebra.  In particular,  $\mathbb{C}^{(M_p)}$ is a topological ring. 
%\end{proposition}

%\begin{remark}\label{B-rem}
%Likewise, applying the proofs of corresponding assertions in \cite{D.V.1}, one has the next assertion.
%\begin{enumerate}
%\item If $P(D)$ is an $(M_p)$-ultradifferentiable linear operator, it is continuous for the above defined sharp topology.
%\item With the definition of sharp topology one can introduce in a natural way the notion of a well-posed problem for the corresponding linear problem (linear equation).
%\end{enumerate}
%\end{remark}

%%
\subsection{Roumieu case}    

The basis of  neighbourhoods of zero in $\mathcal G^{\{M_p\}}(\Omega)$ consists of  sets
\begin{align*}
B^{(k_i)}_{K,(h_i)}=\left\{[(g_{\varepsilon})]\in\mathcal G^{\{M_p\}}(\Omega):\|g_{\varepsilon}e^{{N}_{k_i}({1}/{\varepsilon})}\|_ {K,(h_i)}=O(1)\right\},
\; K\subset\subset \Omega, \; (k_i), (h_i)\in\mathfrak{R}.
\end{align*}

%The same assertion and remark,  as Proposition \ref{B-pro} and Remark \ref{B-rem}, with the substitute of positive  constants $k$ and $h$ with sequences $(k_i)$ and $(h_i)$ in $\mathfrak R$, holds true in the Roumieu case. 
We formulate more general versions with $E$ being  a Fr\'echet algebra.  Now the basis of neighbourhoods for  the Hausdorff sharp  topology is
\begin{align*}
B^{k_i,p}_{K,h_i}=&\left\{ [(g_{\varepsilon})] \in\mathcal G^{\{M_p\}}[C^\infty(\Omega,E)]: \|g_\varepsilon e^{{N}_{(k_i)}(1/\varepsilon)}\|_{K,(h_i),p}=O(1)\right\}, \\
 &\hspace{5cm} K\subset\subset \Omega, \; p\in\mathbb N, \; (k_i), (h_i)\in\mathfrak{R}.
\end{align*}
Note that contrary to what happens in Beurling case we do not have a countable base of neighbourhoods of zero.  %As in Section \ref{BC}, we have  the next proposition and remark.

%\begin{proposition}\label{B-R-pro}
%$\mathcal{G}^{\{M_p\}}[C^\infty(\Omega,E)]$  is a topological module over $\mathcal{G}^{\{M_p\}}(\Omega)$.  Moreover, $\mathcal{G}^{\{M_p\}}[C^\infty(\Omega,E)]$ is  a topological algebra.  In particular, if $E=\mathbb C$, then $\mathbb{C}^{\{M_p\}}$ is a topological ring. 
%\end{proposition}

%%
\subsection{Continuity results}    

Certain non-linear operations can be shown that they are continuous with respect to introduced sharp topologies. We will show this for polynomials.  Let $P$ be a polynomial of the form $P(s)=\sum_{i\leq B}c_is^i, $ $s\in\mathbb R^d, $  i.e.  $P((g_\varepsilon))=(P(g_\varepsilon))=\left(\sum_{i\leq B}c_i g_\varepsilon^i\right), $ $(g_\varepsilon)\in \mathcal E_M^*(\Omega),$ $c_i\in\mathbb C, $ $i=0,\dots,B,$ $B\in\mathbb Z_+.$  Then the mapping $[(g_\varepsilon)]\to [P((g_\varepsilon))], $ $\mathcal G^*(\Omega)\to \mathcal G^*(\Omega)$ is continuous.  Let us show this. There holds
\begin{align*}
P(g_\varepsilon)-P(f_\varepsilon)=(g_\varepsilon-f_\varepsilon) P_0(g_\varepsilon,f_\varepsilon),\;\;\varepsilon\in(0,1),
\end{align*} 
where $P_0$ is a sum of products of $(g_\varepsilon^k)$  and $(f_\varepsilon^s)$,  $s,k\leq B$ ($g=[(g_\varepsilon)], f=[(f_\varepsilon)]$).  Since $(g_\varepsilon-f_\varepsilon)\in \mathcal{N}^*(\Omega)$ the continuity of a polynomial mapping follows.  If $[(g_\varepsilon)]=[(f_\varepsilon)],$ then $[(P(g_\varepsilon)-P(f_\varepsilon))]=0$ in $\mathcal{G}^\ast(\Omega).$

\begin{remark}\label{main}
\begin{enumerate}
\item $\mathcal{G}^{*}[C^\infty(\Omega,E)]$  is a topological module over $\mathcal{G}^{*}(\Omega)$.  Moreover, $\mathcal{G}^{*}[C^\infty(\Omega,E)]$ is  a topological algebra.  In particular, if $E=\mathbb C$, then $\mathbb{C}^{*}$ is a topological ring. 
\item If one considers constant functions $[(c_\varepsilon)],$  $c_\varepsilon(x)=c_\varepsilon\in E,$ $\varepsilon\in (0,1),$ in both cases (which means that in this case $\Omega $ is irrelevant) we obtain strong simplifications: 
\begin{align*}
\mathcal E^{(M_p)}_M[E] & =\left\{(c_\varepsilon): \forall p\in\mathbb N, \;\exists k>0, \;\mu_p(c_\varepsilon)=O(e^{M(k/\varepsilon)})\right\},\\
\mathcal E^{\{M_p\}}_M[E] & =\left\{(c_\varepsilon): \forall p\in\mathbb N, \;\exists (k_i) \in\mathfrak{R}, \;\mu_p(c_\varepsilon)=O(e^{N_{(k_i)}(1/\varepsilon)})\right\},
 \end{align*}
and the corresponding definitions for $\mathcal N^{(M_p)}[E]$, respectively,   $\mathcal N^{\{M_p\}}[E].$ In this way one obtains  $\mathcal G^{(M_p)}[E]=\mathcal{E}^{(M_p)}_M[E]/\mathcal{N}^{(M_p)}[E]$, respectively,  $\mathcal G^{\{M_p\}}[E]=\mathcal{E}^{\{M_p\}}_M[E]/\mathcal{N}^{\{M_p\}}[E].$
\item If $P(D)$ is an $*$-ultradifferentiable linear operator, the operator $P(D)\colon \mathcal{G}^*(\Omega)\to \mathcal{G}^*(\Omega)$ is continuous for the above defined sharp topology.
\item With the definition of sharp topology one can introduce in a natural way the notion of a well-posed problem for the corresponding linear and non-linear equations.
\end{enumerate}
\end{remark}

%\begin{remark}\label{cat} As in Proposition \ref{B-R-pro} we can consider the cases when $\Omega$ is irrelevant i.e.  when we consider polynomials $P:\mathcal G^*[C^\infty(\Omega,E)]\rightarrow\mathcal G^*[C^\infty(\Omega,F)],$ where  $F$ is topologized by a family $\nu_j , $ $j \in \mathbb N$ of seminorms.  Again, if $[(g_\varepsilon-f_\varepsilon)]=0$ in $\mathcal{G}^\ast [C^\infty(\Omega,E)],$ then $[P(g_\varepsilon)-P(f_\varepsilon)]=0$ in $\mathcal{G}^\ast [C^\infty(\Omega,F)].$
%\end{remark}

%%%%%%%
\section{Regular elements and strong association}
\label{s3}
%%%%%%%

We recall the notion of regular $*$-generalized function (where $*$ is $(M_p)$ or $\{M_p\}$ class) \cite{D.V.1}:

A generalized Beurling ultradistribution $[(g_\varepsilon)]$ is an element of $\mathcal{G}^{\infty,(M_p)}(\Omega)$ ($(M_p)$-regular) if it has a representative in  $\mathcal{E}^{\infty,(M_p)}(\Omega)$, where
\begin{align*}
\mathcal{E}^{\infty,(M_p)}(\Omega) =\left\{(g_{\varepsilon}) \in \mathcal{E}^{(M_p)}_M (\Omega): \forall K \subset \subset \Omega,\exists k>0 ,  \forall h>0,\; \|g_{\varepsilon}\|_ {K,h} = O(e^{M({k}/{\varepsilon})})\right\}.
\end{align*}

A generalized Roumieu ultradistribution $[(g_\varepsilon)]$ is an element of $\mathcal{G}^{\infty,\{M_p\}}(\Omega)$ ($\{M_p\}$-regular) if it has a representative in  $\mathcal{E}^{\infty,\{M_p\}}(\Omega)$, where
\begin{align}
\mathcal{E}^{\infty,\{M_p\}}(\Omega)\ = \Big\{&(g_{\varepsilon}) \in \mathcal{E}_M^{\{M_p\}}(\Omega):\nonumber\\
&\label{RDo}
 \forall K \subset\subset \Omega,\   \exists (k_i) \in \mathfrak{R},\;  \forall (h_i) \in \mathfrak{R}, \; \|g_{\varepsilon}\|_ {K,(h_i)} = O(e^{N_{(k_i)}(1/\varepsilon)}) \Big\}.
\end{align}
In both cases we call such elements regular generalized ultradistributions (of corresponding type).

Well-known notion of strong association in Colombeau algebra is easily transferred in the case of Beurling generalized functions. This is not the case for the Roumieu generalized functions.

Beurling case: A generalized ultradistribution $g=[(g_{\varepsilon})]\in\mathcal G^{(M_p)}(\Omega)$ is  strongly $B$-associated to an ultradistribution $T\in\mathcal D'^{(M_p)}(\Omega)$ if for any compact set $K\subset\Omega$, there exists $b >0,$  such that for  any test function $\phi\in\mathcal D^{(M_p)}(K)$, 
\begin{align} \label{bc1}
\left|\left\langle g_{\varepsilon} -T, \phi\right\rangle \right| = O\left(e^{-{M}({b}/{\varepsilon})}\right). 
\end{align}

%  Moreover,  in the Roumieu case, for the proof of regularity in Theorem \ref{r-theorem}, we have to assume an additional assumption.

Roumieu case:
Generalized ultradistribution $g=[(g_{\varepsilon})]\in\mathcal G^{\{M_p\}}(\Omega)$, is  weakly $R$-associated to an ultradistribution $T\in\mathcal D'^{\{M_p\}}(\Omega)$  if for any compact set $K\subset\Omega$ there exists
%%%%%%%%%%%%%%%%%%%
%%%%%%%%%%%%%%%%%%%%here exists 
$(b_i)\in\mathfrak R$ such that for every  $\phi\in\mathcal D^{\{M_p\}}(K),$
% there exist   $C>0$  and $\varepsilon_0\in(0,1),$ such that
\begin{equation}\label{RC}
\left|\left\langle g_{\varepsilon} -T,\phi\right\rangle\right|=O\left(e^{-N_{(b_i)}(1/\varepsilon)}\right).
\end{equation}
The essential difference between the Beurling case and the Roumieu case is that one can not use the arguments of Baire theorem in the Roumieu case. Because of that we introduce the notion  of strong  $R$-association.

Generalized ultradistribution
$g=[(g_{\varepsilon})]\in\mathcal G^{\{M_p\}}(\Omega)$, is strongly $R$-associated to an ultradistribution $T\in\mathcal D'^{\{M_p\}}(\Omega)$  if for any compact set $K\subset\Omega$
there exist   $(\ell_i), (b_i)\in\mathfrak R$,  such that for every $\phi\in\mathcal D^{\{M_p\}}(K),$
\begin{equation}\label{RCpr}
\left|\left\langle g_{\varepsilon} -T,\phi\right\rangle\right|=O\left(e^{-N_{(b_i)}(1/\varepsilon)}\|\phi\|_{K,(\ell_i)}\right).
\end{equation}

\begin{remark}\label{baire}
\begin{enumerate}
\item[(i)]
If  $T$ is compactly supported, i.e., $T\in\mathcal E'^*(\Omega)$,
then in (\ref{bc1}) and (\ref{RCpr}) one can put 
$\phi\in\mathcal E^{(M_p)}(K),$ and 
$\phi\in\mathcal E^{M_p,(\ell_i)}(K),$ respectively. (As we noted, $K$ is always a regular compact set.)
\item[(ii)]
In the Beurling  case,  by the Baire theorem, we can see that the positive real $b$ and $\varepsilon_0$ can be chosen to be the same for all test functions  contained in a neighbourhood of zero in  $\mathcal D^{(M_p)}(K)$. So the equivalent definition in Beurling case is: For any compact subset $K$ of $\Omega$ there exist positive constants $b $ and $\ell $ such that 
\begin{equation}\label{uss} 
|\langle g_{\varepsilon} -T, \phi\rangle| = O\left(e^{-M(b/\varepsilon)} \|\phi\|_{K,\ell}\right),\quad \phi\in\mathcal D^{(M_p)}(K).
\end{equation}
\end{enumerate}
The proof follows from the Baire theorem applied to the sets 
$E_n := \{\phi \in \mathcal D^{(M_p)}(K) : |\langle g_\varepsilon - T, \phi\rangle | \leq e^{ - M( 1 /n\varepsilon)}, \;\varepsilon \leq 1/n\}$ in the complete metric space $\mathcal D^{(M_p)}(K).$ Since $\bigcup_{n\in\mathbb{N}} E_n=\mathcal{D}^{(M_p)}(K)$, there exist $m$, $\ell\in\mathbb{N},$ $\phi_0\in\mathcal{D}^{(M_p)}(K),$ $r>0$ and 
\begin{equation*}
    V_\ell=\left\{\phi\in\mathcal{D}^{(M_p)}(K):\|\phi-\phi_0\|_{K,l}<r\right\},
\end{equation*}
so that $V_\ell\subset E_m.$ 
So, for any $\phi\in \mathcal D^{(M_p)}(K), \phi\neq 0,$  and $\lambda=r/(2||\phi||_{K,\ell})$, there holds
$\lambda\phi+\phi_0\in V_\ell$
and
\begin{equation*}
    |\langle g_\varepsilon-T,\lambda\phi+\phi_0 \rangle| e^{M(1/m\varepsilon)}\leq 1/m, \; |\langle g_\varepsilon-T,\phi_0 \rangle|e^{M(1/m\varepsilon)}\leq 1/m, \;\;\varepsilon\leq 1/m.
\end{equation*}
This, with 
\begin{align*}
|\langle g_\varepsilon-T,\phi \rangle|e^{M(1/m\varepsilon)}
&=\frac{1}{\lambda}|\left\langle g_\varepsilon-T,\lambda\phi+\phi_0 \right\rangle-
\langle g_\varepsilon-T,\phi_0 \rangle|e^{M(1/m\varepsilon)}\\
&\leq \frac{2}{m\lambda}\leq \frac{4}{rm}||\phi||_{K,\ell},
\end{align*}
gives (\ref{uss}).%\end{enumerate}
\end{remark}
Relation between  $(f_\varepsilon)\in\mathcal N^*(\Omega)$ and $(\langle f_\varepsilon,\phi\rangle)\in\mathcal N_{\mathbb C}$ for any $\phi\in\mathcal D^*(\Omega),$ is illustrated in the next example.
\begin{example}\label{e1} Let $g_\varepsilon(t)=e^{it/\varepsilon},$ $t\in \mathbb{R},$ $\varepsilon\in (0,1).$ Clearly $(g_\varepsilon)\notin \mathcal{N}^*$. On the other hand, $(\langle g_\varepsilon,\phi\rangle)\in \mathcal{N}^*_\mathbb{C}$ for every $\phi\in\mathcal D^*(\mathbb R).$ 
Let us show this. For every
$\phi\in\mathcal D^*(\mathbb R),$
\begin{align*}
\langle g_\varepsilon(t),\phi(t)\rangle=\hat\phi(1/\varepsilon),\quad\varepsilon<1.
\end{align*}
Since $ \hat\phi\in \mathcal{S}^*(\mathbb R)$, for every $p>0,$ respectively, $(p_i)\in \mathfrak{R},$ and corresponding $C>0$  and $\varepsilon_0\in(0,1),$ there holds
\begin{align*}
\hat\phi(1/\varepsilon)\leq Ce^{-M(p/\varepsilon)},\;\varepsilon<\varepsilon_0,\;\text{ respectively, }\;\hat\phi(1/\varepsilon)\leq Ce^{-N_{(p_i)}(1/\varepsilon)},\;\varepsilon<\varepsilon_0.
\end{align*}
Thus,  $(\langle g_\varepsilon,\phi\rangle)\in \mathcal{N}^*_\mathbb{C}$ for every $\phi\in\mathcal D^*(\mathbb R).$ 
\end{example}

\begin{remark}\label{r*} 
 Mollifiers $(\phi_\varepsilon)$  of the form $\phi_\varepsilon=1/\varepsilon^d \phi(\cdot/\varepsilon), $ $\varepsilon\in(0,1)$,   $\mathcal F(\phi)\in \mathcal{D}^{(M_p)}(\mathbb{R}^d),$ respectively,  $\mathcal F(\phi)\in \mathcal{D}^{\{M_p\}}(\mathbb{R}^d),$  and $\mathcal{F}(\phi)=1$ in a neighbourhood of zero are introduced and used in \cite{D.V.1}. (Actually, they are defined there in a more general form.) The following assertion is proved in \cite{D.V.1}: If  the image of a Beurling, respectively, Roumieu,  ultradistribution $f$  embedded  into the space of generalized Beurling, respectively, generalized Roumieu ultradistributions, through the convolution $(f*\phi_\varepsilon)$
% $\phi\geq 0$, 
is regular (i.e.  $(f*\phi_\varepsilon)\in\mathcal{E}^{\infty,(M_p)}(\Omega)$, respectively, $(f*\phi_\varepsilon)\in\mathcal{E}^{\infty,\{M_p\}}(\Omega)$), then $f$ is an ultradifferentiable function of the same kind,  i.e.   $f\in\mathcal E^{(M_p)}(\Omega)$, respectively, $f\in\mathcal E^{\{M_p\}}(\Omega)$.
\end{remark}

For the Roumieu case  the next  proposition relates the weak and strong $R$-association.  
%A Roumieu ultradistribution is weakly $R$-associated to its embedding in the space of generalized Roumieu ultradistribution.   
Since  $\Omega\to\mathcal{G}^*(\Omega)$ is a fine and supple sheaf of differential algebras (\cite[Proposition 4.3]{D.V.1}), in the sequel we consider in the proofs compactly supported ultradistributions. 

\begin{proposition}\label{RUM}
\begin{enumerate}
\item[a)] Let $T\in\mathcal E'^{\{M_p\}}(\mathbb R^d)$ and $[(g_\varepsilon)]\in\mathcal G^{\{M_p\}}(\mathbb R^d)$ be given by its regularization $g_\varepsilon= T*\phi_\varepsilon$,  $\varepsilon\in (0,1),$ where $\phi_\varepsilon, $ $\varepsilon \in (0,1)$, is the mollifier as it is stated above.  Then  $T$ and $[(g_\varepsilon)]$ are weakly $R$-associated (cf.  \eqref{RC}).
\item[b)] Assume that in \eqref{RC} the sequence $(b_i)$ has the property that the sequence 
\begin{align*}
M^1_p=M_p\prod_{i=1}^pb_i, \quad p\in\mathbb N,
\end{align*}
also satisfies conditions (M.1) and (M.2) (so $M_p^1$,  $p\in\mathbb N$ satisfies (M.1)-(M.3)).  Then,  there exist  $(h_i)\in\mathfrak R$,  $(\tilde b_i)\in\mathfrak R$,  $C>0$ and $\varepsilon_0\in(0,1)$ such that
\begin{equation*}\label{uslrum}
|\langle g_\varepsilon-T,\theta\rangle|\leq Ce^{-N_{(\tilde b_i)}(1/\varepsilon)}||\theta||_{K,(h_i)}, \quad \varepsilon\in(0,\varepsilon_0),\; \theta\in\mathcal E^{\{M_p\}}(K).
\end{equation*}
\end{enumerate}
\end{proposition}

\begin{proof}
\begin{enumerate}
\item[a)] Since $T\in\mathcal E'^{\{M_p\}}(\mathbb R^d)$ the convolution $T*\phi_\varepsilon$ is well defined. By a suitable transfer of arguments to the Roumieu case in Proposition 5.2 and  Corollary 5.3  of  \cite{D.V.1}, one can infer,   that \eqref{RC} holds.
\item[b)] Let  $M^1(\cdot)$ be the associated function  that corresponds to the sequence $(M^1_p).$ We have
\begin{align*}
\mathcal D'^{\{M_p\}}(\mathbb R^d)\hookrightarrow \mathcal D'^{(M^1_p)}(\mathbb R^d),    \quad \left(\text{because } \mathcal D^{(M_p^1)}(\mathbb R^d)\hookrightarrow \mathcal D^{\{M_p\}}(\mathbb R^d)\right),
\end{align*}
($\hookrightarrow$ dense and continuous embedding). We note that the net of mollifiers $(\phi_\varepsilon)$ in $\mathcal D^{\{M_p\}}$ stated above is also a net of mollifiers in $\mathcal D^{(M^1_p)}(\mathbb R^d).$  So, \eqref{RC} and the embedding theorem of \cite{D.V.1} give that $T$ and $g_\varepsilon=T*\phi_\varepsilon, $ $\varepsilon\in(0,1)$,  are elements  of $ \mathcal D^{(M^1_p)}(\mathbb R^d)$ and $\mathcal G^{(M^1_p)}(\mathbb R^d)$,  respectively,  and  that for every $\theta\in\mathcal D^{(M_p^1)}(\mathbb R^d)$ there holds that 
\begin{align*}
|\langle g_\varepsilon-T,\theta\rangle|\leq Ce^{M^1(1/\varepsilon)}, \quad \varepsilon\in(0,\varepsilon_0),
\end{align*} 
with suitable $C$ and $\varepsilon_0$ which depend on $\theta$. Thus,  $(g_\varepsilon)$ and $T$ are strongly $B$-associated as elements of  $\mathcal D^{(M^1_p)}(\mathbb R^d)$ and $\mathcal G^{(M^1_p)}(\mathbb R^d)$.  So, by Remark \ref{baire}, for any compact subset $K$ of $\Omega$ there exist positive constants $C, $ $\ell$, $t$ and $\varepsilon_0\in(0,1),$ such that 
\begin{equation*}%\label{uss} 
|\langle g_{\varepsilon} -T, \phi\rangle| \leq C e^{-M^1(t/\varepsilon)} \|\phi\|_{K,\ell},\quad \varepsilon\in(0,\varepsilon_0),\;  \phi\in \mathcal E^{(M^1_p)}(K).
\end{equation*}
   
Let us show that there exists a  sequence $(\tilde b_i)\in\mathfrak R$ such that 
\begin{align*}
 M^1(b/\varepsilon)=N_{(\tilde b_i)}(1/\varepsilon), \quad \varepsilon\in(0,\varepsilon_0).
\end{align*}
One has 
\begin{equation}\label{nej1} 
M^1(t/\varepsilon)=\sup_{p\in\mathbb N}\ln \frac{t^p}{\varepsilon^p M_p\prod_{i=1}^pb_i}=\sup_{p\in\mathbb N}\ln \frac{(1/\varepsilon)^p}{ M_p\prod_{i=1}^p(b_i/t)}=N_{(\tilde b_i)}(1/\varepsilon).
\end{equation}
where we put $\tilde b_i=b_i/t, $ $i\in\mathbb N.$  This is a sequence in $\mathfrak R.$ Let $\rho\in\mathcal E^{M^1_p,h}(K).$ Then,
\begin{align*}
\sup_{\alpha\in\mathbb N^d, x\in K}\frac{|\rho^{(\alpha)}(x)|}{h^pM_p\prod_{i=1}^pb_i}=\sup_{\alpha\in\mathbb N^d, x\in K}\frac{|\rho^{(\alpha)}(x)|}{M_p\prod_{i=1}^p(hb_i)}= ||\rho||_{K,(h_i)},
\end{align*}
where $h_i=hb_i, $ $i\in\mathbb N$ is a sequence in $\mathfrak{R}. $ With this and \eqref{nej1} we have 
\begin{align*}
e^{-M^1(t/\varepsilon)}||\rho||_{K,M_p^1,h}= e^{-N_{(\tilde b_i)}(1/\varepsilon)}||\rho||_{K,(h_i)}, \quad \rho\in\mathcal E^{\{M_p\}}(K).
\end{align*}
This completes the proof of the proposition.
\qedhere
\end{enumerate}
\end{proof}

%\noindent Since in the Roumieu case we have uncountable basis for  $\mathcal D^{\{M_p\}}(K)$ (for any $K\subset\subset\Omega$), we have another definition of strong association in the Roumieu case.

The main results of this section are the  next two theorems in which we consider separately Beurling and Roumieu case.

\begin{theorem}\label{b-theorem}
If a regular generalized Beurling  ultradistribution $g=[(g_\varepsilon)]\in\mathcal G^{\infty,(M_p)}(\Omega)$ is strongly $B$-associated to a Beurling  ultradistribution $T$, then $T\in \mathcal E^{(M_p)}(\Omega)$,  i.e.  it is an ultradifferentiable function of Beurling type.
\end{theorem}

\begin{proof}  
%For the proof,  it is enough to show that for any given  $\lambda >0$,  there exists $A>0$,  such that $e^{{M}({|\xi|}/{\lambda})}|\hat{T}(\xi)|, $ $\varepsilon\in(0,\varepsilon_0),$ is bounded for $|\xi|>A.$ 

Since we have the partition of unity in the present non-quasianalytic case,   we can suppose, from now on,  without loss of generality, that 
$\mbox{supp } g_\varepsilon\subset K,$ $\varepsilon\in(0,1),$ where $K\subset\subset \Omega$.  We know by \cite{D.V.1} that compactly supported $g$ belongs 
to  $\mathcal G^{\infty,(M_p)}(\Omega)$ if and only its Fourier transforms $\hat{g}_{\varepsilon},$ $\varepsilon\in (0,1)$ satisfy,
\begin{equation}\label{ft}
\exists k>0, \forall h>0, \exists C_h >0,  \exists \varepsilon_0\in(0,1), \;\sup_{\xi\in\mathbb R^d} |\hat{g}_{\varepsilon}(\xi)|e^{M({|\xi|}/{h})} \leq C_h e^{M({k}/{\varepsilon})}, \;\; \varepsilon\in(0,\varepsilon_0).
\end{equation}  
In order to prove the theorem, by Paley-Wiener theorem, we  have to prove that 
\begin{equation*}
\sup_{\xi\in\mathbb R^d} \left|\hat{T}(\xi)e^{M({|\xi|}/{\lambda})}\right| <\infty,
\end{equation*}
for every $\lambda >0$. Note that $x\mapsto e^{-ix\cdot \xi}\in\mathcal E^{M_p,\ell}(K)$ for any $\ell>0.$  (This function also belongs to  $\mathcal E^{M_p,(\ell_i)}(K)$, for any $(\ell_i)\in\mathfrak R$.) By \eqref{uss} and \eqref{ft}, it follows that there exists constant $C>0$ depending on the compact set $K$ and $\ell>0$ such that,  by direct calculation, one has $\|e^{-ix\xi}\|_{K,\ell}\leq Ce^{M(|\xi|/\ell)}, $ $\xi\in \mathbb R^d,$ and
\begin{align*}
\left|\hat{T}(\xi)- \hat{g_{\varepsilon}}(\xi)\right| = \left|\left\langle T-g_{\varepsilon},e^{-i\xi x}\right\rangle\right| \leq Ce^{{M}(|\xi|/{\ell})-{M}(b/{\varepsilon})},\quad \xi\in\mathbb R^d,\; \varepsilon\in(0,\varepsilon_0).
\end{align*}
Since
\begin{align*}
\left|\hat{T}(\xi)\right| \leq \left|\hat{T}(\xi)- \hat{g_{\varepsilon}}(\xi)\right| +\left|\hat{g_{\varepsilon}}(\xi)\right|,\quad \xi\in\mathbb R^d, \;\varepsilon\in(0,\varepsilon_0),
\end{align*}
by the hypothesis on the regularity of $g$ and by \eqref{ft},  we obtain that  for any given $h>0,$ there holds
\begin{align*}
\left|\hat{T}(\xi)\right| \leq Ce^{{M}({|\xi|}/{\ell})-{M}(b/{\varepsilon})} + C_h e^{{M}({k}/{\varepsilon})-M({|\xi|}/{h})}, \quad \xi\in\mathbb R^d, \;\varepsilon\in(0,\varepsilon_0).
\end{align*}
For any $\lambda>0$ there exist $C_1>0$ and $C_2>0$ depending on $K$ and $h>0$, such that
\begin{align*}
e^{M({|\xi|}/{\lambda})}\left|\hat{T}(\xi)\right| \leq C_1e^{{M}({|\xi|}/{\lambda})+{M}({|\xi|}/{\ell})-M(b/\varepsilon)} +  C_2 e^{{M}({|\xi|}/{\lambda})+{M}({k}/{\varepsilon})-{M}({|\xi|}/{h})}, 
\end{align*}
for every $\xi\in\mathbb R^d, $ $\varepsilon\in(0,\varepsilon_0).$ Now, the proof is finished if we manage to choose $\varepsilon$ and $h$ such that for any $|\xi|>A$ (which will be determined below),  there hold
\begin{align}\label{jed2}
M({|\xi|}/{\lambda})+M({|\xi|}/{l})-M(b/{\varepsilon}) \leq 0 \;\;\mbox{ and }\;\; M({|\xi|}/{\lambda})+M({k}/{\varepsilon})-M({|\xi|}/{h})\leq 0. 
\end{align}
Let  us show this.  Choose $A>0$ and $\delta>0$ such that for given $\varepsilon_0<1$ it holds $b\delta/|\xi|=\varepsilon<\varepsilon_0,$ $\xi>A.$ Moreover, shrink $\delta>0,$  if it is necessary, so that
\begin{align*}
M({|\xi|}/{\delta}) \geq M({|\xi|}/{\lambda})+M({|\xi|}/{l}), \quad |\xi|>A.
\end{align*}
This is possible because of the properties of the associated function $M$.  So, for $\varepsilon=b\delta/|\xi|<\varepsilon_0,$  the first desired inequality \eqref{jed2} holds. We replace in the second inequality $1/\varepsilon$ by $b|\xi|/\delta$ and find $h$ such that,
\begin{align*}
M(|\xi|/h) \geq M(|\xi|/\lambda)+M(|kb\xi|/\delta),  \quad |\xi|>A.
\end{align*}
Thus, the second inequality holds and the proof is finished.
\end{proof}

\begin{remark} 
For Colombeau generalized functions an analogous result concerning $\mathcal{C}^{\infty}$  regularity was proved in \cite{Sca1}, for analytic regularity in \cite{psv.1}, and for finite type regularities in  \cite{psv.2}.
\end{remark}

Now we consider the Roumieu case.
\begin{theorem}\label{r-theorem}
Let $g$ be a regular generalized Roumieu ultradistribution, $g=[(g_\varepsilon)]\in\mathcal G^{\infty,\{M_p\}}(\Omega),$ i.e. let \eqref{RDo} hold. Assume that $g$ is strongly $R$-associated to a Roumieu  ultradistribution $T,$ i.e. that \eqref{RCpr} holds. Moreover, we assume that $b_i\leq k_i,$ $i\in\mathbb{N}.$ 
Then $T$ is a Roumieu   ultradifferentiable function,  i.e.  $T\in\mathcal E^{\{M_p\}}(\Omega).$
\end{theorem}

\begin{remark}\label{comp}
We note that condition \eqref{RCpr} is  weaker than the one given in the definition of $\mathcal N^{\{M_p\}}(\Omega).$ This condition does not allow the decrease of $(b_i)$ while  condition \eqref{RDo} does not allow the increase of $(k_i)$. But,  decreasing the sequence $(k_i),$ one can assume in Theorem \ref{r-theorem} that
\begin{equation*}
    (b_i)=(k_i).
\end{equation*}

Following Remark \ref{r*}, we note  that for a given $\psi\in\mathcal D^{\{M_p\}}(\mathbb R^n),$ the convolution with the mollifier  
$(\phi_\varepsilon)$ (given in that remark) satisfies $(\psi*\phi_\varepsilon-\psi)\in\mathcal N^{\{M_p\}}(\mathbb R^n)$. Moreover, one can   see that
$(\psi*\phi_\varepsilon)$  satisfies conditions (\ref{RDo}) and (\ref{RCpr}). Let $\mu\in\mathcal D^{\{M_p\}}(K)$. Then,
$(\psi*\phi_\varepsilon+e^{-N_{(b_i)}(1/\varepsilon)}\mu)$ satisfies (\ref{RDo}).
So we can conclude that Theorem \ref{r-theorem} is related to  perturbations of $(T*\phi_\varepsilon)$ so that the assertion of that theorem holds true. 
\end{remark}

\begin{proof}
We again suppose, without loss of generality, that both $T$ and $g_\varepsilon,$ $\varepsilon\in (0,1),$ have compact supports included in $ K\subset\subset\Omega.$  
Note that $k_i\geq b_i,$ $i\in \mathbb{N},$ implies 
\begin{align}\label{afo}
\frac{c_{(k_i)}(\rho)}{c_{(b_i)}(\rho)}\leq 1\quad\Rightarrow\quad c^{-1}_{(b_i)}(\rho)\leq c^{-1}_{(k_i)}(\rho)\quad\Rightarrow\quad c_{(k_i)}(c^{-1}_{(b_i)}(\rho))\leq \rho,\quad\rho\geq 0.
\end{align} 
This condition will be used later since it implies that for any $(\delta_i)\in\mathfrak{R}$ there holds that $c_{(k_i)}(c^{-1}_{(b_i)}(c_{(\delta_i)}(\rho))),$ $\rho>0,$ is a subordinate function.

By the  Paley-Wiener theorem  for ultradistributions with compact support (cf. \cite{Komatsu1}), the Roumieu regularity of $T$ amounts that  
\begin{equation}\label{**}
\sup_{\xi\in\mathbb R^d}\left| \hat{T}(\xi) e^{N_{(\lambda_i)}(|\xi|)}\right| <\infty\;\;\text{ for every }\;(\lambda_i)\in\mathfrak R.
\end{equation}
So let us prove \eqref{**}.  For $|T(\xi)-\hat g_\varepsilon(\xi)|$ we apply the definition of the strong $R$-association  \eqref{RCpr} and for $|\hat g_\varepsilon(\xi)|$ we apply \eqref{RDo}.  Since $x\mapsto e^{-ix\cdot \xi}\in\mathcal E^{M_p,(\ell_i)}(K)$, for every $(\ell_i)\in\mathfrak R,$ we calculate  the norm $\|e^{-ix\cdot\xi}\|_{K,(\ell_i)}$ and obtain $||e^{-ix\cdot \xi}||_{K,(\ell_i)}\leq Ce^{N_{(\ell_i)}(1/\varepsilon)},$ $\varepsilon\in(0,\varepsilon_0).$ With $|\hat{T}(\xi)| \leq  |\hat{g_{\varepsilon}}(\xi)|+|\hat{T}(\xi)- \hat{g_{\varepsilon}}(\xi)|,$ $\varepsilon\in (0,1),$ applying \eqref{RCpr} and the above norm estimate of $e^{-ix\cdot\xi},$ we obtain, for $\xi\in\mathbb R^d,$ $\varepsilon\in(0,\varepsilon_0),$ and suitable constants $C_1,C_2>0,$%for given $(\lambda_i)$by the assumption, 
\begin{align}\label{final} 
e^{N_{(\lambda_i)}(|\xi|)}&\left|\hat{T}(\xi)\right| \leq\nonumber\\
& C_1 e^{N_{(\ell_i)}(|\xi|)+N_{(\lambda_i)}(|\xi|)-N_{(b_i)}(1/\varepsilon)} + C_2 e^{N_{(k_i)}(1/\varepsilon)+N_{(\lambda_i)}(|\xi|)-N_{(h_i)}(|\xi|)}.
\end{align} 
Choose $(\delta_i)$ in $\mathfrak R$ so that
\begin{align*}
N_{(\delta_i)}(|\xi|)\geq N_{(\ell_i)}(|\xi|)+N_{(\lambda_i)}(|\xi|), \quad \xi\in\mathbb R^d.
\end{align*}
With this, we have to estimate
\begin{align*}
e^{N_{(\delta_i)}(|\xi|)-N_{(b_i)}(1/\varepsilon)},\quad |\xi|\geq A, \;\varepsilon\in(0,\varepsilon_0),
\end{align*}
for enough large $A>0$ which will be determined later.
Recall the notation with  subordinate functions,%  which corresponds to $(\delta_i)$
\begin{equation*}%\label{rnej}
N_{(\delta_i)}(|\xi|)=M(c_{(\delta_i)}(|\xi|)) \;\; \mbox{ and }\;\; N_{(b_i)}(1/\varepsilon)= M(c_{(b_i)}(1/\varepsilon)).
\end{equation*}
%Choose $A>0$ and $\delta>0$ such that for given $\varepsilon_0<1$ it holds $\frac{b\delta}{|\xi|}=\varepsilon<\varepsilon_0,$ $\xi>A.$
Let $|\xi|>A$ and $\varepsilon$ be defined by
\begin{align}\label{eq1}
c_{(\delta_i)}(|\xi|)=c_{(b_i)}(1/\varepsilon)\; \;\mbox{ i.e.  } \; \; \frac{1}{\varepsilon} =(c_{(b_i)})^{-1}(c_{(\delta_i)}(|\xi|)).
\end{align}
Since $(c_{(b_i)})^{-1}(c_{(\delta_i)}(|\xi|)), $ $\xi\in \mathbb{R}^d,$ is an increasing continuous function (not subordinated one, in general), for enough large $A$  and $|\xi|>A$, one obtains that $\varepsilon\in(0,\varepsilon_0).$ So $A$ is determined by the condition that in \eqref{eq1} $\varepsilon\in (0,\varepsilon_0).$ With this, we obtain that the first therm of the right hand side in \eqref{final} is bounded.  Now we insert this $\varepsilon$ into 
\begin{align*}
N_{(k_i)}(1/\varepsilon)=M(c_{(k_i)}(1/\varepsilon))=M(c_{(k_i)}((c_{(b_i)})^{-1}(c_{(\delta_i)}(|\xi|)))), \quad \varepsilon\in(0,\varepsilon_0).
\end{align*}
By \eqref{afo},% $c_{(k_i)}(c_{(b_i)})^{-1}$ is a subordinate function and
\begin{align*}
c_{(t_i)}(\rho)=c_{(k_i)}((c_{(b_i)})^{-1}(c_{(\delta_i)}(\rho))),\quad \rho\geq 0,
\end{align*}
is also a  subordinating function.
%All the subordinate functions are monotonically increasing.
The first step in the estimation of the second therm in \eqref{final} is to choose $(\theta_i)\in\mathfrak R$ so that $N_{(\theta_i)}(|\xi|)\geq N_{(\lambda_i)}(|\xi|)+N_{(k_i)}(|\xi|), $ $\xi\in\mathbb{R}^d.$ Now we choose $(h_i)$, thus $c_{(h_i)},$  so that $M(c_{(h_i)}(\rho))\geq N_{(\theta_i)}(\rho),$ $ \rho>0$. So with  $\varepsilon\in(0,1),$ one can choose $(h_i)$ such that the second therm in \eqref{final} is bounded for $\xi\in\mathbb{R}^d$, as well. This  finishes the proof.
\end{proof}

%%%%%%%%%%%
\section{Equalities}
%%%%%%%%%%%

In the next three theorems we show some  important properties of generalized ultradistributions, the ones which are proved for Colombeau generalized functions in the paper \cite{PSV}. The ideas of the proofs are similar to the ones from that paper but the framework of generalized ultradistributions  is more complex.

%We need the parametrix in the ultradistribution spaces $\mathcal D'^*(\mathbb R^d)$.
%In our non-quasianalytic case, it can be deduced from \cite{Komatsu1} and 
%\cite {Komatsu2}, but we give the proof based on \cite{ppv}.
%By Theorem 2.1 and Proposition 2.2 in \cite{ppv} we have:
%\begin{itemize}
%\item[I a)] in Beurling case: there exists $P(x)$ of class $(M_p),$ $P(x)\neq 0,$ $x\in \mathbb{R}^d$ such that $P(D)\mathcal{F}^{-1}(1/P(x))=\delta$ where $P$ satisfies
%\begin{align}\label{par1}
%\forall r>0,\; \exists k>0,\; \exists C>0,\; \forall x\in \mathbb{R}^d,\; \forall \alpha\in \mathbb{N}_0^d,\;\;\left|D^\alpha\left(\frac{1}{P(x)}\right)\right|\leq \frac{C\alpha!}{r^\alpha} e^{-M(kx)};
%\end{align}
%\item[I b)] in Roumieux case: there exists $P(x)$ of class $\{M_p\},$ $P(x)\neq 0,$ $x\in \mathbb{R}^d$ such that $P(D)\mathcal{F}^{-1}(1/P(x))=\delta$ where $P$ satisfies
%\begin{align}\label{par2}
%\forall r>0,\;\exists (k_p)\in\mathfrak{R},\;\exists C>0,\;\forall x\in \mathbb{R}^d,\;\forall \alpha\in \mathbb{N}_0^d,\;\;\left|D^\alpha\left(\frac{1}{P(x)}\right)\right|\leq \frac{C\alpha!}{r^\alpha} e^{-N(kx)}.
%\end{align}
%\end{itemize}

We will use so called Komatsu parametrix, see \cite[p. 195]{Komatsu2} (also \cite{ppv}). Let $K$ be a compact set so that $0\in K^\circ.$ Let $P_r(D),$ respectively $P_{(r_i)}(D),$ be ultradifferential operators of the form given in Section \ref{sec1.3}. Then in Beurling case, there exist $G\in \mathcal{D}^{M_p,\tilde r}_{K}(\mathbb{R}^d)$ and $\psi\in \mathcal{D}^{(M_p)}_K(\mathbb R^d)$, respectively $G\in \mathcal{D}^{M_p,(\tilde r_i)}_{K}(\mathbb{R}^d)$ and $\psi\in \mathcal{D}^{\{M_p\}}_K(\mathbb R^d)$ in Roumieu case, such that 
\begin{align*}
P_{\tilde r}(D)G=\delta+\psi,\;\text{ respectively }\;P_{(\tilde r_i)}(D)G=\delta+\psi.
\end{align*}

Actually, we have in \cite{ppv} weaker conditions on $(M_p)$ but for the next theorem we need $(M3)$ as well.

\begin{theorem}\label{thm 4.1}
Let a generalized ultradistribution $[(f_\varepsilon)]\in\mathcal{G}^{(M_p)}(\Omega),$ respectively, $[(f_\varepsilon)]\in \mathcal{G}^{\{M_p\}}(\Omega)$ be such that $\mbox{supp }f_\varepsilon\subset K,$ $\varepsilon\in (0,1)$ and 
\begin{itemize}
\item in Beurling case: $\exists \tilde r>0,$ $\forall \phi\in \mathcal D^{M_p,\tilde r}(K),$ respectively,
\item in Roumieu case: $\exists (\tilde r_i)\in\mathfrak{R},$ $\forall \phi\in \mathcal D^{M_p,(\tilde r_i)}(K),$
\end{itemize}
\begin{align}\label{nulcon}
\left(\int_{\mathbb{R}^d}f_\varepsilon(x)\phi(x)\ dx\right)\in \mathcal{N}^{(M_p)}_\mathbb{C},\text{ respectively }\in \mathcal{N}^{\{M_p\}}_\mathbb{C}.
\end{align}
Then $f=0$ in $\mathcal{G}^{(M_p)}(\Omega),$ respectively, $f=0$ in $\mathcal{G}^{\{M_p\}}(\Omega).$
\end{theorem}

\begin{proof}
In order to simplify the proof, we assume that  $\Omega=\mathbb R^d$. Let $\Theta$ be a bounded open set  in $ \mathbb R^d$ such that $K\subset\subset\Theta$ and let $R>0$ be enough large so that   $\Theta-K\subset B(0,R),$ where $B(0,R)$ denotes a closed ball with center zero and radius $R$.  By the Komatsu type parametrix,
\begin{align*}
f_\varepsilon*\delta=f_\varepsilon=P^*(D)f_\varepsilon*G+f_\varepsilon*\theta,
\end{align*}
where $P^*$ stands for $P_{\tilde r},$ respectively $P_{(\tilde r_p)},$
and by \eqref{nulcon} we have that for every $k>0,$ respectively, $(k_i)\in\mathfrak R$, for every $\phi\in\mathcal D^{M_p,\tilde r}(B(0,R))$,  $\mbox{supp } \phi\in B(0,R),$ respectively,  $\phi\in\mathcal D^{M_p,(\tilde r_i)}(B(0,R))$,  $\mbox{supp } \phi\in B(0,R)$ there exist $C=C_\phi>0$ and  $\varepsilon_0=\varepsilon_{0,\phi}\in(0,1)$ such that
\begin{equation*}%\label{t1 **}
e^{M(k/\varepsilon)}\left|\int_{B(0,2R)} f_\varepsilon(x)\phi(x)\ dx\right|\leq C,\quad \varepsilon\in (0,\varepsilon_0),
\end{equation*}
respectively,
\begin{equation*}%\label{1t1 **}
e^{N_{(k_i)}(1/\varepsilon)}\left|\int_{B(0,2R)} f_\varepsilon(x)\phi(x)\ dx\right|\leq C,\quad \varepsilon\in (0,\varepsilon_0).
\end{equation*}
Note,  $\mathcal D^{M_p,\tilde r}(K)$ and $\mathcal D^{M_p,(\tilde r_i)}(K)$ are Banach spaces.
Let $k>0,$ respectively $(k_i)\in\mathfrak R,$ be fixed and $(\varepsilon_\nu)_\nu$ be a sequence in $(0,1)$ which strictly decrease to zero.
Put
\begin{equation*}
A_{b,\nu}=\left\{ \phi\in \mathcal D^{M_p,\tilde r}(K): 
e^{M(k/\varepsilon)}\left|\int_{B(0,2R)} f_\varepsilon(x)\phi(x)\ dx\right|\leq1/\varepsilon_\nu,\;\; \varepsilon\in (0,\varepsilon_{\nu})\right\}, \;\nu\in\mathbb N,
\end{equation*}
respectively,
\begin{equation*}
A_{r,\nu}=\left\{ \phi\in \mathcal D^{M_p,(\tilde r_i)}(K): 
e^{N_{(k_i)}(1/\varepsilon)}\left|\int_{B(0,2R)} f_\varepsilon(x)\phi(x)\ dx\right|\leq1/\varepsilon_\nu,\;\; \varepsilon\in (0,\varepsilon_{\nu})\right\},\; \nu\in\mathbb N.
\end{equation*}
 By the Baire theorem  and the fact that
 \begin{equation*}
\bigcup_{\nu\in\mathbb N}A_{b,\nu}= \mathcal D^{M_p,\tilde r}(K),\; \mbox{ respectively, }\;\bigcup_{\nu\in\mathbb N}A_{r,\nu}= \mathcal D^{M_p,(\tilde r_i)}(K),
 \end{equation*}
we have that there exists
$\nu_0, \phi_0$ and  $V$, a neighbourhood of zero in  $\mathcal D^{M_p,\tilde r}(B(0,R))$, respectively, $\mathcal D^{M_p,(\tilde r_i)}(B(0,R))$  such that
\begin{equation*}
\phi_0+V\subset A_{b,{\nu_0}} \Rightarrow e^{M(k/\varepsilon)}\left|\int_{B(0,2R)} f_\varepsilon(x)(\phi_0(x)+\phi(x))dx\right|\leq 1/\varepsilon_{\nu_0},\;\; \phi\in V,\;  \varepsilon\in (0,\varepsilon_{\nu_0}),
\end{equation*}
respectively,
\begin{equation*}
 \phi_0+V\subset A_{r,{\nu_0}} \Rightarrow e^{N_{(k_i)}(1/\varepsilon)}\left|\int_{B(0,2R)} f_\varepsilon(x)(\phi_0(x)+\phi(x))dx\right|\leq  1/\varepsilon_{\nu_0},\;\; \phi\in V,\;  \varepsilon\in (0,\varepsilon_{\nu_0}).
\end{equation*}

Thus, there exists $C>0$ and $\varepsilon_0\in(0,1)$ such that
\begin{equation*}
e^{M(k/\varepsilon)}\left|\int_{B(0,2R)} f_\varepsilon(x)\phi(x)\ dx\right|\leq C,\quad \phi\in V,\;  \varepsilon\in (0,\varepsilon_0),
\end{equation*}
respectively,
\begin{equation*}
e^{N_{(k_i)}(1/\varepsilon)}\left|\int_{B(0,2R)} f_\varepsilon(x)\phi(x)\ dx\right|\leq C,\quad \phi\in V,\;  \varepsilon\in (0,\varepsilon_0).
\end{equation*}
Since, for $x\in K,$ $u\in \Theta,$ $x-u\in B(0,R)$ it follows that $\{G(u-\cdot): u\in\Theta\}$ and $\{\theta(u-\cdot): u\in \Theta\},$ are bounded sets in  $\mathcal D^{M_p,t}(B(0,R))$, respectively,  $\mathcal D^{M_p,(t_i)}(B(0,R)).$ So  they are absorbed by $V.$ 
This implies that for every $k>0,$ respectively, $(k_i)\in\mathfrak R$  there exist $C>0$ and $\varepsilon_0\in (0,1)$ such that for $\varepsilon\in(0,\varepsilon_0),$ $u\in \Theta,$
\begin{align*}%\label{Gteta}
|f_\varepsilon\ast G(u)|,|f_\varepsilon\ast \theta(u)|\leq C e^{-M(k/\varepsilon)},
\mbox{ respectively, } \leq C e^{-N_{(k_i)}(1/\varepsilon)}.
\end{align*}
Since $(f_\varepsilon\ast G)\in \mathcal{E}^{(M_p)}_M(\Theta),$ respectively,  $(f_\varepsilon\ast G)\in \mathcal{E}^{\{M_p\}}_M(\Theta),$  by \cite[Proposition 4.2]{D.V.1}, we have that $(f_\varepsilon\ast G) \in \mathcal{N}^{(M_p)}(\Theta)$, respectively,  $(f_\varepsilon\ast G)\in \mathcal{N}^{\{M_p\}}(\Theta)$ and so $(P(D)(f_\varepsilon\ast G))\in \mathcal{N}^{(M_p)}(\Theta)$,  respectively,  $(P(D)(f_\varepsilon\ast G))\in \mathcal{N}^{\{M_p\}}(\Theta).$  The same proposition of \cite{D.V.1} implies $(f_\varepsilon\ast \theta)\in \mathcal{N}^{(M_p)}(\Theta),$ respectively, $(f_\varepsilon\ast \theta)\in \mathcal{N}^{\{M_p\}}(\Theta).$ 

Thus, $(f_\varepsilon)\in \mathcal{N}^{(M_p)}(\Theta),$ respectively, $(f_\varepsilon)\in \mathcal{N}^{\{M_p\}}(\Theta)$. 

Let us note that in the case of general $\Omega$ one has to use for $K\subset\subset\Omega$ an appropriate open set $\Omega'\subset\Omega$ instead of the ball $B(0,R)$ and  the  partition of unity.  
\end{proof}

\begin{theorem}\label{4-2}
Let $\Omega\subset \mathbb{R}^d$ and $f\in \mathcal{G}^{\infty,\ast}(\Omega).$ Then $f=0$ if and only if 
%$\left(\int f_\varepsilon(t)\kappa(t)dt\right) \in\mathcal N^\ast$, 
there exists $\varepsilon_0\in (0,1)$ such that for every $h>0$, respectively, for every $(h_i)\in\mathfrak{R},$ and every $\kappa\in \mathcal{D}^\ast(\Omega),$ exists $C=C_{h,\kappa}>0,$ respectively, $C=C_{(h_i),\kappa}>0,$ such that
\begin{align}\label{ops*}
&\int f_\varepsilon(t)\kappa(t)dt\; e^{M(h/\varepsilon)}<C,\quad \varepsilon<\varepsilon_0,\quad %\nonumber\\
\text{respectively,}\\ \nonumber
&\int f_\varepsilon(t)\kappa(t)dt \;e^{N_{(h_i)}(1/\varepsilon)}<C,\quad \varepsilon<\varepsilon_0.
\end{align}

\end{theorem}

\begin{remark} Condition \eqref{ops*} is stronger than $\left(\int f_\varepsilon(t)\kappa(t)dt\right) \in\mathcal N^\ast_\mathbb{C}$, for every $\kappa\in \mathcal{D}^\ast(\Omega).$ Actually, in the Beurling case we need to assume weaker condition than \eqref{ops*}. We will show this in Proposition \ref{p6} below. It is an open question whether we can in the Roumieu case assume a weaker condition than that in \eqref{ops*}.
\end{remark}
\begin{proof}
We assume that $\mbox{supp }f_\varepsilon\subset K,$ $ \varepsilon\in (0,1),$ 
where $K$ is a compact set of $\Omega.$ Our aim is to show that the conditions of Theorem \ref{thm 4.1} hold. We have 
\begin{align*}%\label{f51}
&\exists s>0,\; \forall h>0,\; \exists \tilde\varepsilon_0\in(0,1),\; \exists T_h=T>0,\;\; \left\| f_\varepsilon e^{-M(s/\varepsilon)}\right\|_{\mathcal{E}^{M_p,h}(K)} <T,\quad \varepsilon\in (0,\tilde\varepsilon_0),
\end{align*}
respectively,
\begin{align*}%\label{f52}
&\exists (s_i),\; \forall (h_i),\; \exists \tilde\varepsilon_0\in(0,1),\; \exists T_{(h_i)}=T>0,\;\; \left\| f_\varepsilon e^{-N_{(s_i)}(1/\varepsilon)}\right\|_{\mathcal{E}^{M_p,(h_i)}(K)} <T,\quad \varepsilon\in (0,\tilde\varepsilon_0).
\end{align*}
(We assume that $\varepsilon_0<\tilde\varepsilon_0$.) Condition \eqref{ops*} is equivalent with: there exists $\varepsilon_0$ such that for for every $h>0,$ respectively $(h_i)\in \mathfrak{R},$ 
\begin{align*} \label{nov11}
e^{M(h/\varepsilon)-M(s/\varepsilon)}\left\langle f_\varepsilon,\varphi\right\rangle=O(1),\;\; \varphi\in\mathcal D^{(M_p)}(\Omega),\text{ respectively, }\\
e^{N_{(h_i)}(1/\varepsilon)-N_{(s_i)}(1/\varepsilon)}\langle f_\varepsilon,\varphi\rangle=O(1),\;\; \varphi\in\mathcal D^{\{M_p\}}(\Omega),\;\varepsilon<\varepsilon_0.
\end{align*}
 This follows from the  properties of associated functions $M$ and $N_{(k_i)}.$

Let $\varepsilon<\varepsilon_0$. We note that for every $h>0$, respectively, $(h_i)$ there holds that
$e^{M(h/\varepsilon)-M(s/\varepsilon)}\langle f_\varepsilon, \psi \rangle$, respectively, $e^{N_{(h_i)}(1/\varepsilon)-N_{(s_i)}(1/\varepsilon)}\langle f_\varepsilon, \psi \rangle,$
is well defined for every $\psi\in \mathcal{D}^{M_p,h},$
respectively, for every $\psi\in \mathcal{D}^{M_p,(h_i)}.$

Next, fix $h>0$, respectively $(h_i)\in\mathfrak R.$  Since Banach spaces $\mathcal{D}^{M_p,h}(K)$ and $\mathcal{D}^{M_p,(h_i)}(K)$ are 
barreled and the set 
$\{e^{M(h/\varepsilon)-M(s/\varepsilon)}f_\varepsilon : \varepsilon<\varepsilon_0\}$, respectively 
$\{e^{N_{(h_i)}(1/\varepsilon)-N_{(s_i)}(1/\varepsilon)}f_\varepsilon:\varepsilon<\varepsilon_0\}$, is weakly bounded, it follows that it is strongly bounded. 
So, there exists open set $V_h=\{\varphi\in\mathcal{D}^{M_p,h}: \|\varphi\|_{K,h}<\delta\},$ respectively, open set $V_{(h_i)}=\{\varphi\in\mathcal{D}^{M_p,(h_i)}: \|\varphi\|_{K,(h_i)}<\delta\}$ such that 
\begin{align*}
\exists C>0,\;\forall \varphi\in V_h,\quad e^{M(k/\varepsilon)}\left\langle f_\varepsilon e^{-M(s/\varepsilon)},\varphi\right\rangle \leq C,\;\varepsilon<\varepsilon_0,
\end{align*}
respectively,
\begin{align*}
\exists C>0,\;\forall \varphi\in V_{(h_i)},\quad e^{N_{(h_i)}(1/\varepsilon)}\left\langle f_\varepsilon e^{-N_{(s_i)}(1/\varepsilon)},\varphi\right\rangle \leq C,\;\varepsilon<\varepsilon_0.
\end{align*}
We continue the proof in Beurling case. Let $(\varphi_n)_n$ be a sequence in $\mathcal{D}^{(M_p)}$ so that $\varphi_n\to \varphi$ in $\mathcal{D}^{M_p,h}$ ($\varphi\in \mathcal{D}^{M_p,h}$). There holds
\begin{align*}
e^{M(h/\varepsilon)}\left\langle f_\varepsilon e^{-M(s/\varepsilon)},\varphi_n-\varphi_m\right\rangle\leq C,\quad m,n>n_0, \;\varepsilon<\varepsilon_0.
\end{align*}
 This implies 
\begin{align*}
\left|\left\langle e^{M(h/\varepsilon)-M(s/\varepsilon)}f_\varepsilon,\varphi_n-\varphi\right\rangle\right|\leq C,\quad n>n_0,\; \varepsilon<\varepsilon_0.
\end{align*}
Thus for every $\varepsilon<\varepsilon_0$,
$
|\langle e^{M(h/\varepsilon)-M(s/\varepsilon)}f_\varepsilon,\varphi\rangle|\leq C.
$

If above instead of $M(h/\varepsilon),$ $M(s/\varepsilon)$ and $V_h$ one puts 
$N_{(h_i)}(1/\varepsilon),$ 
$N_{(s_i)}(1/\varepsilon)$ 
and $V_{(h_i)},$ one obtains that there exists $(h_i)\in \mathfrak{R}$ such that for every $\varphi\in \mathcal{D}^{M_p,(h_i)}$
\begin{align*}
\left|\left\langle e^{N_{(h_i)}(1/\varepsilon)-N_{(s_i)}(1/\varepsilon)}f_\varepsilon,\varphi\right\rangle\right|\leq C,\quad \varepsilon<\varepsilon_0.
\end{align*}
Thus, the conditions of the previous theorem holds and the assertion follows.
\end{proof}

%\begin{remark}
%We note that under the conditions of Theorems \ref{b-theorem} and \ref{r-theorem}, respectively,  for $(g_\varepsilon)\in \mathcal E_M^*(\mathbb R^d)$ (and $T\in\mathcal E'^*(\mathbb R^d)$),  the weak negligible condition for $(g_\varepsilon)$ of Theorem \ref{4-2} directly implies that  $(g_\varepsilon)\in\mathcal N^*(\mathbb R^d).$
%\end{remark}

\begin{proposition}\label{p6}
In the Beurling case one can assume 
\begin{align*}
\forall h>0,\;\;\forall \kappa\in\mathcal{D}^{(M_p)}(K),\;\;&\exists \varepsilon_0\in (0,1),\;\;\exists C=C_{h,\kappa}>0,\\
&\int f_\varepsilon(t)\kappa(t)dt\;e^{M(h/\varepsilon)}\leq C,\quad \varepsilon<\varepsilon_0.
\end{align*}
\end{proposition}

\begin{proof}
Since $\mathcal{D}^{(M_p)}(K)$ is a complete metric space and 
\begin{equation*}
F_n=\left\{\phi\in \mathcal{D}^{(M_p)}(K): \left|\int g_\varepsilon (x)e^{-M(n\varepsilon)}\phi(x)dx\right|\leq 1, \; \varepsilon\leq 1/n \right\},
\end{equation*}
are closed so that $\bigcup_{n\in\mathbb{N}}F_n=\mathcal{D}^{(M_p)}(K),$ one of them, $F_{n_0},$ contains an open ball. So there exist $s>0,$ $V_{K,k},$ $\varepsilon_0\in (0,1)$ and $C>0$ such that
\begin{align*}
\phi\in V_{K,k}\;\wedge\;\varepsilon<\varepsilon_0\;\Rightarrow\;\left\langle f_\varepsilon e^{-M(s/\varepsilon)},\phi\right\rangle\leq C.
\end{align*}
Since any element in $\mathcal D^{(M_p)}(K)$ is absorbed by $V_{K,k},$ the condition of Theorem \ref{4-2} holds in the Beurling case and as in this theorem we finish the proof.
\end{proof}

\begin{theorem}
Let $f\in \mathcal{G}^{*}(\mathbb{R}^d)$ be such that $f(\cdot+h)-f(\cdot)\in\mathcal N^{*}$  for any $h\in \mathbb{R}^d.$ Then $f=C=[(C)_\varepsilon]\in \mathbb{C}^{*}.$
\end{theorem}

\begin{proof}
We know in Beurling case:
\begin{align*}
&\forall R\in \mathbb{N},\;\forall p\in \mathbb{N},\;\forall x\in \mathbb{R}^d,\; \exists l_x\in \mathbb{N},
\\
\sup_{t\in B(0,R)}& e^{M(p/\varepsilon)} \left|f_\varepsilon(x+t)-f_\varepsilon(t)\right|\leq 1,\;  \varepsilon<1/\ell_x,
\end{align*}
while in Roumieu case:
\begin{align*}
&\forall R\in \mathbb{N},\; \forall (p_i)\in \mathfrak R,\; \forall x\in \mathbb{R}^d,\; \exists l_x\in \mathbb{N},\\
\sup_{t\in B(0,R)} & e^{N_{(p_i)}(1/\varepsilon)} \left|f_\varepsilon(x+t)-f_\varepsilon(t)\right|\leq 1, \; \varepsilon<1/l_x.
\end{align*}

 Our aim is to show that $[(f_\varepsilon(t))]=[(f_\varepsilon(0))]$ for every $t\in\mathbb{R}$. 
 
 We will prove the assertion now in Roumieu case since by substituting relevant sequences in $\mathfrak R$ with the corresponding positive constants, one arrives to the proof of Beurling case.  Fix $(p_i)\in\mathfrak R$. Let $l\in \mathbb{N}$ and 
\begin{align*}
F_{(p_i),l}=\left\{x\in \mathbb{R}^d:\; \varepsilon<1/l\; \Rightarrow\; \sup_{t\in B(0,R)}e^{N_{(p_i)}(1/\varepsilon)}|f_\varepsilon(x+t)-f_\varepsilon(t)|\leq 1 \right\}.
\end{align*}
As   $\bigcup_{l=1}^\infty F_{(p_i),l} =\mathbb R^d$, by the Baire theorem, there exist $x_0\in \mathbb{R}^d,$ $l_0\in \mathbb{N}$ and $r>0,$ so that  $B(x_0,r)\subset F_{(p_i),l_0}.$ 
%Let $r_1\leq r.$
If $x\in B(0,R-r),$ $h\in B(0,r)$ then for $\varepsilon<1/l_0,$
\begin{align*}
\sup_{t\in B(0,R)} & e^{N_{(p_i)}(1/\varepsilon)}|f_\varepsilon(t+h)-f_\varepsilon(t)|\\
\leq & \sup_{t\in B(0,R-r)}e^{N_{(p_i)}(1/\varepsilon)}\left(|f_\varepsilon(x+t+h)-f_\varepsilon(t+h)|+|f_\varepsilon(x+t+h)-f_\varepsilon(t)|\right)\\
%\leq & \sup_{t\in B(0,R-r)}e^{N_{(p_i)}(1/\varepsilon)}\left(|f_\varepsilon(x+t)-f_\varepsilon(t+h)|+|f_\varepsilon(x+t+h)-f_\varepsilon(t)|\right)\\ 
\leq & 2.
\end{align*} 
Thus, for every $(p_i)\in\mathfrak R$,  there exists $C>0$ such that
\begin{align}\label{jed}
\sup_{t\in B(0,R)}|f_\varepsilon(t+h)-f_\varepsilon(t)|\leq C e^{-N_{(p_i)}(1/\varepsilon)}.
\end{align}
Let $f_\varepsilon(0)=C_\varepsilon,$ $\varepsilon\in(0,1).$ Estimate \eqref{jed} implies that for $|h|<r,$ $f_\varepsilon(h)=C_\varepsilon+Ce^{-N_{(p_i)}(1/\varepsilon)},$ $\varepsilon\in(0,1),$ which means that $[(f_\varepsilon(h))]=[(C_\varepsilon)].$ Up to any point $t$ of $\mathbb{R}^d$ one can come with a finite number of steps of length $h_i,$ $h_i<r,$ $i=1,\dots,s.$ So $[(f_\varepsilon(0))]=[(f_\varepsilon(h_1))]=\dots=[(f_\varepsilon(h_{s-1}))]=[(f_\varepsilon(t))].$
 This completes the proof of the assertion.
\end{proof}

%%%%%%%%%%%%%%%%%%%%%%%%%%%%%%%%%%%%%%%%%%%%
\bmhead{Acknowledgments}
%\section*{Acknowledgement}

The work of S. Pilipovi\'c is supported by the project F10 of the Serbian Academy of Sciences and Arts.\\
The work of M. \v Zigi\' c was partially supported by the Science Fund of the Republic of Serbia, \#GRANT No 2727, \textit{Global and local analysis of operators and distributions} - GOALS, and by the Ministry of Science, Technological Development and Innovation of the Republic of Serbia,  project no.  451-03-47/2023-01/200125.
%%%%%%%%%%%%%%%%%%%%%%%%%%%%%%%%%%%%%%%%%%%

%%%%%%%%%%

%%%%%%%%%%%%%

\begin{thebibliography}{99}
%%%%%%%%%%

\bibitem{col1}
Colombeau,  J.F.:
Elementary Introduction to New Generalized Functions.
North-Holland,  Amsterdam (1985)

\bibitem{D.V.1}
Debrouwere,  A., Vernaeve,  H., Vindas, J.:
Optimal embeddings of ultradistributions into differential algebras.
Monatsh.  Math.  186, 407--438 (2018).  
https://doi.org/10.1007/s00605-017-1066-6

%\bibitem{Del}, A. Delcroix, Regular rapidly decreasing nonlinear generalized functions. Application to microlocal regularity, J. Math. Anal. Appl. 327 (2007), 564?584.

\bibitem{Del-Has-Pil-Val}
Delcroix, A., Hasler,  M.,  Pilipovi\'c, S.,  Valmorin,  V.:
Embeddings of ultradistributions and periodic hyperfunctions in Colombeau type algebras through sequence spaces.
Math. Proc. Cambridge Philos. Soc. 137, 697--708  (2004).  
https://doi.org/10.1017/S0305004104007923

\bibitem{DHPV} 
Delcroix,  A., Hasler,  M.,  Pilipovi\'c, S., Valmorin,  V.:
Sequence spaces with exponent weights. Realizations of Colombeau type algebras. Dissertationes Math.  447 (56 p.) (2007).
https://doi.org/10.4064/dm447-0-1

%\bibitem{Del-Sca} A. Delcroix and D. Scarpal\'ezos, Asymptotic scales-asymptotic algebras, Integral Transforms Spec. Funct. 6 (1998), 181?190.

%\bibitem{Del-Sca-2}  A. Delcroix and D. Scarpalezos,, Sharp topologies on (C,E,P)-algebras, in: M. Grosser et al. (eds.), Nonlinear Theory of Generalized Functions, Chapman and Hall/CRC Res. Notes Math. 401, 1999, 165?173.

\bibitem{Del-Sca-3} 
Delcroix, A.,  Scarpal\'ezos, D.:
Topology on asymptotic algebras of generalized functions and applications.
Monatsh. Math. 129,  1--14 (2000).
https://doi.org/10.1007/s006050050001

\bibitem{gkos}
Grosser,  M.,  Kunzinger, M., Oberguggenberger,  M.,  Steinbauer, R.:
Geometric Generalized Functions with Applications to General Relativity.
Kluwer,  Dordrecht (2001)

\bibitem{her-kunz} 
H\"ormann, G.,  Kunzinger,  M.:
Microlocal properties of basic operations in Colombeau algebras.
J. Math. Anal. Appl. 261,  254--270 (2001).
https://doi.org/10.1006/jmaa.2001.7498

\bibitem{Komatsu1} 
Komatsu,  H.:
Ultradistributions, I: Structure theorems and a characterization.
J. Fac. Sci. Univ. Tokyo Sect. IA Math. 20,  25--105  (1973)

\bibitem{Komatsu2} 
Komatsu,  H.:  
Microlocal analysis in Gevrey classes and in complex domains.  In: Cattabriga, L., Rodino, L. (eds) Microlocal analysis and applications, pp. 161--236, Springer, Berlin (1991).
https://doi.org/10.1007/BFb0085124

\bibitem{KMMP}  
Maksimovi\'c, S.,  Mincheva-Kami\' nska,  S.,  Pilipovi\'c, S.,  Sokoloski,  P.:
A sequential approach to ultradistribution spaces. 
Publ. Inst. Math. (Beograd) (N.S.) 100, 17--48  (2016).
https://doi.org/10.2298/PIM1614017M

\bibitem{ober 001}
Oberguggenberger, M.:
Multiplication of Distributions and Application to Partial Differential Equations.
Pitman Res. Notes Math. Ser. 259, Longman, Harlow (1992)

\bibitem{ps} Pilipovic,  S.,  Scarpaelezos, D.:
Colombeau generalized ultradistributions.
Math. Proc. Camb. Phil. Soc.  130,  541--553 (2001).
https://doi.org/10.1017/S0305004101005072

\bibitem{psv.1} 
Pilipovi\' c,  S.,  Scarpal\' ezos, D.,  Valmorin,  V.:
Real analytic generalized functions. 
Monatsh. Math. 156,  85--102  (2009).
https://doi.org/10.1007/s00605-008-0524-6

\bibitem{PSV} 
Pilipovi\' c,  S.,  Scarpal\' ezos,  D., Valmorin,  V.:
Equalities in algebras of generalized functions.
Forum Math. 18, 789--801  (2006).
https://doi.org/10.1515/FORUM.2006.039

\bibitem{psv.2} 
Pilipovi\' c,  S.,  Scarpal\' ezos,  D.,  Vindas,  J.:
Besov regularity in non-linear generalized functions.
Monatsh. Math. 201,  483--498  (2023).
https://doi.org/10.1007/s00605-022-01783-1

\bibitem{ppv} 
Pilipovi\' c,  S.,  Prangoski,  B.,  Vindas,  J.:
On quasianalytic classes of Gelfand-Shilov type. Parametrix and convolution.
J. Math. Pures Appl.  116, 174--210   (2018).
https://doi.org/10.1016/j.matpur.2017.10.008

\bibitem{Sca1} 
Scarpal\' ezos,  D.:
Colombeau's Generalized Functions: Topological Structures; Microlocal Properties. A Simplified Point of View.
Prepublication Mathematiques de Paris 7/CNRS, URA212 (1993)

\bibitem{Sca2} 
Scarpal\' ezos,  D.:
Some remarks on functionality of Colombeau's construction: topological and microlocal aspects and applications.
Integral Transforms Spec.  Funct. 6,  295--307  (1998).
https://doi.org/10.1080/10652469808819174

\bibitem{Sca3} 
Scarpal\' ezos,  D.:
Colombeau's generalized functions: Topological structures, microlocal properties. A simplified point of view, Part I. 
Bull. Cl. Sci. Math. Nat. Sci. Math.  25,  89--114  (2000)

\bibitem{Sca4} 
Scarpal\' ezos,  D.:
Colombeau's generalized functions: Topological structures, microlocal properties. A simplified point of view, Part II. 
Publ. Inst. Math. (Beograd) (N.S.) 76, 111--125  (2004).
https://doi.org/10.2298/PIM0476111S

\end{thebibliography}
\end{document}